\documentclass[reqno]{amsart}
\RequirePackage[l2tabu, orthodox]{nag}

\usepackage{amssymb}
\usepackage[style=numeric]{biblatex}
\usepackage{bbold}
\usepackage{bm}
\usepackage{caption}
\usepackage{color}
\usepackage{enumitem}
\usepackage{esint}
\usepackage{mathrsfs}
\usepackage{mathtools}
\usepackage{microtype}
\usepackage{calrsfs}
\usepackage{stmaryrd}
\usepackage{tcolorbox}
\tcbuselibrary{theorems}
\usepackage{hyperref}

\theoremstyle{plain}
\newtheorem{theorem}{Theorem}[section]
\newtheorem{corollary}[theorem]{Corollary}
\newtheorem{proposition}[theorem]{Proposition}
\newtheorem{lemma}[theorem]{Lemma}
\newtheorem{conjecture}[theorem]{Conjecture}

\theoremstyle{definition}
\newtheorem{definition}[theorem]{Definition}

\theoremstyle{remark}

\newtheorem{remark}[theorem]{Remark}

\newcommand{\IGNORE}[1]{}
\newcommand{\C}{\mathbb{C}}
\newcommand{\D}{\mathbb{D}}
\newcommand{\N}{\mathbb{N}}
\newcommand{\R}{\mathbb{R}}
\renewcommand{\S}{\mathbb{S}}

\newcommand{\CRIT}{\mathrm{Crit}}

\newcommand{\DIST}{\mathrm{dist}}

\newcommand {\supp}{\mathrm{supp}}
\newcommand{\DIAM}{\mathrm{diam}}
\newcommand{\tr}{\mathrm{tr}}

\numberwithin{equation}{section}

\title{Avoidance of non-strict saddle points by blow-up}

\author[E.\ M.\ Achour]{El Mehdi Achour}
\author[U.\ L.\ Hryniewicz]{Umberto L.\ Hryniewicz}
\author[M.\ Westdickenberg]{Michael Westdickenberg}

\address{%
  El Mehdi Achour,
  Mathematics of Information Processing,
  RWTH Aachen University,
  Pont\-driesch 12-14,
  D-52062 Aachen,
  Germany}
\email{achour@mathc.rwth-aachen.de}

\address{%
  Umberto L.\ Hryniewicz,
  Geometry and Analysis,
  RWTH Aachen University,
  Pontdriesch 10-12,
  D-52062 Aachen,
  Germany}
\email{hryniewicz@mathga.rwth-aachen.de}

\address{%
  Michael Westdickenberg,
  Institute for Mathematics,
  RWTH Aachen University,
  Im Süsterfeld 2,
  D-52072 Aachen,
  Germany}
\email{mwest@instmath.rwth-aachen.de}

\date{\today}


\begin{document}

\begin{abstract}
It is an old idea to use gradient flows or time-discretized variants thereof as methods for solving minimization problems. In some applications, for example in machine learning contexts, it is important to know that for generic initial data, gradient flow trajectories do not get stuck at saddle points. There are classical results concerned with the non-degenerate situation. But if the Hessian of the objective function has a non-trivial kernel at the critical point, then these results are inconclusive in general. In this paper, we show how relevant information can be extracted by ``blowing up'' the objective function around the non-strict saddle point, i.e., by a suitable non-linear rescaling that makes the higher order geometry visible.
\end{abstract}

\maketitle

\setcounter{tocdepth}{2}

\tableofcontents

\section{Introduction}

Gradient flows are well-known as important special cases of dynamical systems. Despite a long and rich history of mathematical research, new applications and mathematical questions and insights continue to emerge. Recently, the use of gradient flows and their time-discretized siblings as building blocks in optimization procedures, most notably in machine learning contexts, has reignited interest in the convergence behavior of these flows, especially in nonconvex regimes. Since in these applications the initial data is often chosen randomly, an important issue is to identify criteria that ensure that \emph{generic} gradient flow trajectories do not ``get stuck'' at saddle points, but instead flow into (local) minima of the objective function. In this paper we are strictly concerned with finite-dimensional manifolds, and generic is meant in the measure-theoretic sense. We also discuss relevant examples in machine learning.

\subsection{Basic definitions and statement of main problem}

We work under the convention that manifolds are Hausdorff, second countable and finite-dimensional. Let $(M,g)$ be a smooth connected Riemannian manifold, and let
\begin{equation}
\label{function_f}
f:M\to\R
\end{equation}
be a smooth function. The differential of $f$ is denoted by $df$ and its set of critical points by $\CRIT(f)=\{w\in M : df(w)=0\}$. For any~$w\in M$ the unit sphere at~$w$ is denoted by $S_wM=\{v\in T_wM:g(v,v)=1\}$. 

The gradient vector field $\nabla f$ is determined by $df=g(\nabla f,\cdot)$. The flow of $-\nabla f$ will be called the gradient flow of $f$, in spite of the minus sign. This is the unique smooth map $\varphi:\mathcal{D}\to M$, where $\mathcal{D}\subset \R\times M$ is open, $\{0\}\times M \subset \mathcal{D}$, for all $w_0\in M$ the open set 
\[
J_{w_0} = \{t\in\R:(t,w_0)\in\mathcal{D}\}
\]
is the maximal interval of definition of the solution $w(t)$ of the initial value problem 
\[
\dot w=-\nabla f\circ w, \qquad w(0)=w_0
\]
and $\varphi(t,w_0)=w(t)$. We write $\varphi^t(w_0)$ instead of~$\varphi(t,w_0)$. Given $C\subset\CRIT(f)$, consider the~sets
\begin{equation*}
% \label{stable_sets}
\begin{aligned}
W^s_{-}(C) &= \left\{w_0\in M : [0,+\infty) \subset J_{w_0}, \ \inf _{t\geq0} \ \DIST(\varphi^t(w_0),C) = 0  \right\}, \\
W^s(C) &= \left\{w_0\in M : [0,+\infty) \subset J_{w_0}, \ \lim_{t\to+\infty} \DIST(\varphi^t(w_0),C) = 0 \right\}, \\
W^s_{+}(C) &= \left\{w_0\in M : [0,+\infty) \subset J_{w_0}, \ \lim_{t\to+\infty} \varphi^t(w_0) = w_* \ \text{for some} \ w_*\in C \right\}.
\end{aligned}
\end{equation*}
We call the sets $W^s_-(C)$, $W^s(C)$ and $W^s_+(C)$ the \textbf{weak stable}, \textbf{stable} and \textbf{strong stable} sets of $C$, respectively. There are inclusions
\[
C\subset W_+^s(C)\subset W^s(C)\subset W^s_-(C).
\]
If $w_*$ is an isolated point of $\CRIT(f)$ and $C=\{w_*\}$ then all these sets coincide: $W_+^s(\{w_*\})= W^s(\{w_*\})= W^s_-(\{w_*\})$.

We say that \textbf{$f$ attains its value at $w_*\in M$ with order $0\leq k\leq\infty$} if for every smooth curve $\gamma:(-1,1)\to M$ satisfying $\gamma(0)=w_*$, all derivatives up to order $k$ of $t\mapsto f(\gamma(t))-f(w_*)$ vanish at $t=0$. Equivalently, in any choice of local coordinates around $w_*$ all partial derivatives of $f-f(w_*)$ of order $\leq k$ vanish at $w_*$. In particular, $w_*$ is a critical point precisely when $f$ attains its value at~$w_*$ with order~$1$. If the value of $f$ at $w_*$ is attained with order $k\geq j$, then it is also attained with order~$j$.

If $f$ attains its value at $w_*$ with order $k-1$, then there is a well-defined homogeneous polynomial $P:T_{w_*}M \to \R$ of order $k$ given by $P(v) = \frac{d^{k}}{dt^{k}}f\circ\gamma(t)|_{t=0}$ where $\gamma:(-1,1)\to M$ is any smooth curve satisfying $\dot\gamma(0)=v$. The value $P(v)$ does not depend on the choice of~$\gamma$, and
\begin{equation}
\label{expansion_f_defn_P}
f(\exp(v)) = f(w_*) + P(v) + \varepsilon(v), \quad \varepsilon(v) = O(\|v\|_g^{k+1}) \ \text{as} \ v\in T_{w_*}M, \ v\to0.
\end{equation}
holds, where $\|v\|_g = \sqrt{g(v,v)}$ is the induced norm and $\exp$ is the exponential map of $(M,g)$. In fact, $P$ is characterized by~\eqref{expansion_f_defn_P}. Moreover, $P(v)=Q(v,\dots,v)$ for some symmetric $\R$-valued $k$-linear form $Q$ on $T_{w_*}M$. If $k=1$ then $Q=P$ is the differential of $f$ at~$w_*$. If $k=2$ then $Q$ is the \textbf{Hessian} of $f$ at $w_*$, and will be denoted by $D^2f(w_*)$. 

Let $w_*\in\CRIT(f)$. The \textbf{Morse index} $\mu(w_*)\in\{0,\dots,d=\dim M\}$ is defined as the maximal dimension of a linear subspace $V\subset T_{w_*}M$ where $D^2f(w_*)$ is negative-definite, i.e. $D^2f(w_*)(v,v)<0$ holds for all $v \in V \setminus\{0\}$. If $g$ is used to represent $D^2f(w_*)$ as a $g$-symmetric linear map $\nabla^2f(w_*):T_{w_*}M\to T_{w_*}M$ by $g(\nabla^2f(w_*)\cdot,\cdot)=D^2f(w_*)(\cdot,\cdot)$, then $\mu(w_*)$ is the number of negative eigenvalues of $\nabla^2f(w_*)$ counted with multiplicities. We say that $w_*$ is a \textbf{non-degenerate} critical point if $D^2f(w_*)$ is a non-degenerate bilinear form, in the sense that for each $v\in T_{w_*}M$, $v\neq0$,  there exists $v'\in T_{w_*}M$ such that $D^2f(w_*)(v,v')\neq0$. Otherwise, we say that $w_*$ is a \textbf{degenerate} critical point. The \textbf{nullity} of $w_*$ is $\nu(w_*)=\dim \ker D^2f(w_*)$, where $\ker D^2f(w_*)$ is the set of vectors $v\in T_{w_*}M$ such that $D^2f(v,v')=0$ for all $v'\in T_{w_*}M$. Note that $\ker D^2f(w_*)=\ker\nabla^2f(w_*)$, and that~$w_*$ is non-degenerate if, and only if, $\nu(w_*)=0$.

One calls $w_*\in\CRIT(f)$ a \textbf{saddle point} of~$f$ if $w_*$ is neither a local minimum nor a local maximum of~$f$. The saddle point $w_*$ is called \textbf{strict} if $\mu(w_*)>0$. Otherwise it is called a \textbf{non-strict} saddle point. Note that a non-strict saddle point is necessarily a degenerate critical point.

The general problem which motivates this paper can be stated as follows.

\bigskip

\noindent \textbf{Avoidance Problem.} Given $C\subset\CRIT(f)$, is it possible to decide whether ``most'' gradient trajectories avoid $C$? More precisely, can we decide whether $W^s_+(C)$, $W^s(C)$ or $W^s_-(C)$ have zero measure?

\bigskip

The above problem is fundamental. But stated as above, it serves only as a guideline. One does not expect to solve it in this level of generality since the dynamics of a gradient flow can be wild. Already in the case that $C$ consists of a single critical point $w_*$ not much is known, in general, about the structure of the corresponding stable sets. If~$w_*$ is a strict local maximum or minimum then the Avoidance Problem is clear. It is also clear in the non-degenerate case: $W^s_+(\{w_*\})=W^s(\{w_*\})=W^s_-(\{w_*\})$ is a smooth disk of dimension $d-\mu(w_*)$, hence, has zero measure if, and only if, $\mu(w_*)>0$. It turns out that the assumption $\mu(w_*)>0$ guarantees that $W^s(\{w_*\})$ has zero measure even when the non-degeneracy assumption is dropped. This is well-known and can be proved, for instance, with the help of the Splitting Lemma by Gromoll and Meyer~\cite{GromollMeyer}, but will also follow independently from our main results. The study of the much harder case of a non-strict saddle point is, of course, the main goal of our paper.

% \newpage

\subsection{Main results}

We start by defining the appropriate class of saddle points we work with.

\begin{definition}
\label{defn_weak_saddle}
A saddle point $w_* \in \CRIT(f)$ is called \textbf{weakly strict} if:
\begin{itemize}
\item $f$ does not attain its value at $w_*$ with infinite order.
\item Let $2\leq k<\infty$ be the largest integer such that $f$ attains its value at $w_*$ with order $k-1$, and let $P:T_{w_*}M\to\R$ be the homogeneous polynomial of degree $k$ so that $f(\exp(v))=f(w_*)+P(v)+\varepsilon(v)$, where $\varepsilon(v)=O(\|v\|_g^{k+1})$ as $v\to0$. Let $p$ be the restriction of $P$ to $S_{w_*}M$. Every $z \in \CRIT(p)$ satisfying $p(z)\geq0$ has positive Morse index.
\end{itemize}
A weakly strict saddle point $w_*$ is \textbf{tamed} if $\CRIT(p)$ consists of isolated points.
\end{definition}

The motivation for the definition above is that many non-strict saddle points are weakly strict. For an easy and abundant source of examples, let $d\geq2$, $k\geq 3$ and consider a homogeneous polynomial $P$ of degree $k$ in $\R^d$ that attains positive and negative values, and such that every critical point of $p=P|_{\S^{d-1}}$ in $p^{-1}([0,+\infty))$ has positive Morse index. Let $Q:\R^d\to\R$ be smooth, $m>k$ be even and $c\in\R$. Then $0\in\R^d$ is a non-strict saddle point of $K(x)=c+P(x)+\|x\|^{m}Q(x)$ that is weakly strict; here $\|x\|$ denotes the Euclidean norm of $x$ in $\R^d$. If $P$ is so that $p$ is Morse and $p^{-1}([0,+\infty))$ contains no local minima then $0$ is a tamed weakly strict saddle point of $K$. A direct calculation shows that $(x,y,z)\mapsto\frac{1}{2}(xyz-1)^2$ fits this situation. Another generous and important source of examples in the context of linear neural networks is given by Proposition~\ref{prop_lnn_app} below.

\begin{remark}
Strict saddle points are always weakly strict. In fact, in the above notation, $k=2$ when $w_*$ is a strict saddle point and the Hessian $P=D^2f(w_*)$ is a quadratic form that attains negative values. Represent $P(v)=g(Av,v)$ for a linear $g$-symmetric operator $A:T_{w_*}M\to T_{w_*}M$. Then $\CRIT(p)$ consists precisely of the eigenvectors of $A$ normalized to lie on $S_{w_*}M$, the critical values being the corresponding eigenvalues. In particular, $\lambda_-=\inf p <0$ is an eigenvalue. If $v_*\in\CRIT(p)$ and $p(v_*)\geq0>\lambda_-$, then the Hessian of $p$ at $v_*$, being precisely $P|_{(\R v_*)^\bot}$ and represented by $A|_{(\R v_*)^\bot}$, has $\lambda_-$ as an eigenvalue. The spectral theorem for symmetric linear maps was used here. Hence, $v_*$ has positive Morse index as a critical point of $p$ and the strict saddle point $w_*$ checks the requirements of Definition~\ref{defn_weak_saddle}. 
\end{remark}

Our first result reads as follows.

\begin{theorem}
\label{main_thm_larger_sets}
Let $C \subset \CRIT(f)$ be such that every $w_* \in C$ is a tamed weakly strict saddle point. Then $W^s_+(C)$ has zero measure.
\end{theorem}

\begin{corollary}
\label{cor_main}
If $w_*$ is a tamed weakly strict saddle point, and is isolated as a critical point, then $W^s_-(\{w_*\})$ has zero measure.
\end{corollary}

We shall prove Theorem~\ref{main_thm_larger_sets} as a consequence of the following result.

\begin{theorem}
\label{thm_weak_center_stable}
Assume that $x_0 \in \CRIT(f)$ is a tamed weakly strict saddle point. There exists a set $W(x_0)\subset M$ with the following properties.
\begin{itemize}
\item[(a)] For some open neighborhood $B$ of $x_0$, if $w_0$ is such that $\varphi^t(w_0)$ is defined and satisfies $\varphi^t(w_0) \in B$ for all $t\geq 0$, then $w_0 \in W(x_0)$.
\item[(b)] $W(x_0)$ has zero measure.
\end{itemize}
\end{theorem}

\begin{proof}
[Proof of Theorem~\ref{main_thm_larger_sets}]
The proof consists of an elementary covering argument combined with Theorem~\ref{thm_weak_center_stable}. Denote $E=W^s_+(C)$, for simplicity. For each $t\in \R$ the map $\varphi^t$ is defined on the open set $\mathcal{D}_t=\{w\in M:(t,w)\in\mathcal{D}\}$. Let $\Phi = \varphi^1$, $U = \mathcal{D}_1$. The $n$-fold iteration $\Phi^n$ of the map $\Phi$ is defined on an open set $U_n$ defined inductively as $$ U_1=U \qquad U_n = \Phi^{-1}(U_{n-1}) = \{ w \in U : \Phi(w) \in U_{n-1}\}. $$
Note that $\Phi^n$ defines a diffeomorphism of $U_n$ onto the open set $\Phi^n(U_n)$. The pre-image of a set $A$ by $\Phi^n$ is denoted by $\Phi^{-n}(A) = \{ w\in U_n : \Phi^n(w)\in A\}$. Since $\Phi^n$ is a diffeomorphism, if $A$ has Lebesgue measure equal to zero then so does $\Phi^{-n}(A)$.

For each $w_* \in C$ let $W(w_*)$ be the set associated to $w_*$ given by Theorem~\ref{thm_weak_center_stable}.
According to (a) in Theorem~\ref{thm_weak_center_stable}, each $w_*$ has an open neighborhood $B(w_*)$ with the following property: if $w_0\in B(w_*)$ is such that $\varphi^t(w_0)$ is defined and belongs to $B(w_*)$ for all $t\geq 0$, then $w_0 \in W(w_*)$.
Let $\{B(w^i_*)\}_{i\in \N}$ be a countable open cover of $C$ by such balls.
Define $$ E_0 = \bigcup_i \bigcup_n \Phi^{-n}\big(W(w_*^i)\big). $$
Each $\Phi^{-n}(W(w_*^i))$ has zero measure since so does each $W(w_*^i)$.
It follows that $E_0$ has zero measure.
Let $w_0 \in E$ be arbitrary, and let $w_* = \lim_{t\to+\infty}\varphi^t(w_0) \in C$.
There exists $i\in \N$ such that $w_* \in B(w_*^i)$.
Hence we find $T>0$ such that $t\geq T \Rightarrow \varphi^t(w_0) \in B(w_*^i)$.
Hence $n\geq T \Rightarrow \Phi^n(w_0) \in W(w_*^i)$ in view of (a) in Theorem~\ref{thm_weak_center_stable}.
We have shown that $E_0 \supset E$. It follows that $E$ has zero measure.
\end{proof}

Our main result reads as follows.

\begin{theorem}
\label{main_thm_real_analytic_larger_sets}
Assume that $(M,g)$ and $f: M \to \R$ are real-analytic, and also that $f$ is non-constant. Let $C$ be a compact connected component of $\CRIT(f)$ consisting of tamed weakly strict saddle points. Then $W^s_-(C)$ has zero measure.
\end{theorem}

\begin{proof}
We present a proof based on Theorem~\ref{main_thm_larger_sets} and on the well-known Lemma~\ref{lemma_omega_limit_set_one_pt}. Denote $E=W^s_+(C)$ and $E'=W^s_-(C)$ for simplicity. Since by Theorem~\ref{main_thm_larger_sets}, $E$ has Lebesgue measure equal to zero, it suffices to show that $E' \subset E$.

Let $w_0 \in E'$. Denote by $T_+ \in (0,+\infty]$ the maximal existence time of $\varphi^t(w_0)$. We claim that $T_+=+\infty$. If $T_+<+\infty$ then standard ODE theory implies $\varphi^t(w_0)$ escapes every compact subset of $M$ as $t\to T_+$. Since $C$ is compact, $\liminf_{t\to T_+}\DIST(\varphi^t(w_0),C)>0$, in contradiction to the definition of~$E'$. This proves the claim.

Denote the $\omega$-limit set of $t\mapsto \varphi^t(w_0)$ by $\omega(w_0)$ (Definition~\ref{defn_omega_limit_sets}). Then $\omega(w_0)$ is a closed set contained in $\CRIT(f)$. By the definition of $E'$ we must have $\omega(w_0) \cap C \neq \emptyset$, in particular $\omega(w_0) \neq \emptyset$.

We claim that $\omega(w_0)$  is compact. Let $\omega_*$ be a connected component of $\omega(w_0)$ containing a point in $\omega(w_0) \cap C$. Hence $\omega_* \cup C$ is a connected subset of $\CRIT(f)$ that contains the connected component~$C$. It follows that $\omega_* \cup C = C \ \Rightarrow \ \omega_* \subset C$. If $\omega(w_0)$ is disconnected then $\omega_*$ is not compact (Remark~\ref{rmk_omega_sets}), in contradiction to the assumed compactness of~$C$. Hence $\omega(w_0)=\omega_*$ is connected and $\omega(w_0) \subset C$, thus proving the claim.

So far we have that $\omega(w_0)$ is compact, connected and non-empty. Hence $\varphi^{[0,+\infty)}(w_0)$ is bounded. Lemma~\ref{lemma_omega_limit_set_one_pt} applies and we find $w_* \in C$  such that $\omega(w_0)=\{w_*\}$, i.e. $\varphi^t(w_0) \to w_*$ as $t\to+\infty$. Hence $w_0 \in E$.
\end{proof}

% \begin{remark}
% To illustrate the method, consider $f:\R^3\to\R$, $f(x,y,z)=\frac{1}{2}(xyz-1)^2$. The origin $(0,0,0)$ is a critical point where the Hessian vanishes, and is neither a local minimum nor a local maximum. Hence, $(0,0,0)$ is a non-strict saddle point. The function attains its value $\frac{1}{2}$ with order $2$ at the origin, and the first non-zero homogeneous polynomial of the Taylor expansion is $P(x,y,z)=-xyz$. One readily sees that $p=P|_{\S^2}$ has critical points at $(\pm1,0,0)$, $(0,\pm1,0)$, $(0,0,\pm1)$ and $\frac{1}{\sqrt{3}}(\pm1,\pm1,\pm1)$. Moreover, a direct application of Proposition~\ref{prop_lnn_app} implies that $(0,0,0)$ is a weakly strict saddle point. Hence, it is a tamed weakly strict saddle point. Theorem~\ref{main_thm_larger_sets} implies that $W^s_+(\{(0,0,0)\})$ has zero measure.
% \end{remark}

We end this paragraph with a conjecture which is key to our approach to the avoidance problem. Theorem~\ref{thm_weak_center_stable} and its consequences, namely Theorem~\ref{main_thm_larger_sets} and Corollary~\ref{cor_main}, and also our main result Theorem~\ref{main_thm_real_analytic_larger_sets},
conjecturally hold without the tameness assumption. Proposition~\ref{prop_lnn_app} below reveals the wide range of applications that the validity of this conjecture would unleash for linear neural networks. 

\begin{conjecture}
\label{conjecture_bold}
Let $(M,g)$ be real-analytic, and $f:M\to\R$ be real-analytic and non-constant. If a compact connected component $C$ of $\CRIT(f)$ consists of weakly strict saddle points, then $W^s_-(C)$ has zero measure. Moreover, if $M$ is compact and $C\subset\CRIT(f)$ is any compact subset, then $W^s_-(C)$ has zero measure.
\end{conjecture}

Our current results may already find applications to other problems in machine learning. This is the content of future work.

\subsection{Linear neural networks}

For each pair $(r,s)\in\N\times\N$ denote by $\R^{r\times s}$ the space of real matrices with $r$ rows and $s$ columns. The Euclidean inner-product on $\R^{r\times s}$ can be written as $\left<A,B\right>=\tr AB^T$, where $B^T$ denotes the transpose of $B$, and the Euclidean norm as $\|A\| = \sqrt{\tr AA^T}$. 

Let $d_1,\dots,d_{N+1}\in\N$, $x_1,\dots,x_m \in \R^{d_1} \simeq \R^{d_1\times 1}$ and $y_1,\dots,y_m\in\R^{d_{N+1}} \simeq \R^{d_{N+1}\times 1}$. Let $X\in\R^{d_1\times m}$ be the matrix whose $i$-th column is $x_i$, and $Y\in\R^{d_{N+1}\times m}$ be the matrix whose $i$-th column is $y_i$. Consider the polynomial function 
\begin{equation*}
f:\R^{d_2\times d_1} \times \R^{d_3\times d_2} \times \dots \times \R^{d_{N+1}\times d_{N}} \to \R
\end{equation*}
defined by
\begin{equation}
\label{fnct_obj_linear_neural_network}
f(W_1,\dots,W_N) = \frac{1}{2} \| W_N\cdots W_1X - Y\|^2, \qquad W_i \in \R^{d_{i+1}\times d_i}.
\end{equation}

The motivation is as follows. The matrix $W=W_N\dots W_1$ is a linear neural network with weights $W_j$. It feeds from an input $x\in\R^{d_1}\simeq\R^{d_1\times 1}$ to give the output $y = Wx \in \R^{d_{N+1}}\simeq \R^{d_{N+1}\times 1}$. The matrices $X$, $Y$ are the training data: one would like $W$ to satisfy $Wx_i=y_i$ for all $i$, and $f$ measures how much this fails. The function $f$ is called the loss function. We train the model by evolving $(W_1,\dots,W_N)$ under the gradient flow to minimize the loss. As the flow evolves, the linear neural network achieves a progressively better fit to the training data.

Assume that $f$ is non-constant. The variable in $$ V = \R^{d_2\times d_1} \times \R^{d_3\times d_2} \times \dots \times \R^{d_{N+1}\times d_{N}} $$ will be denoted as $\vec{W}=(W_1,\dots,W_N)$. Define $\zeta(\vec{W})\in\{0,\dots,N\}$, $\kappa(\vec{W})\geq1$ by:
\begin{itemize}
    \item $\zeta(\vec{W}) = \#\{i\in\{1,\dots,N\} : W_i=0\}$.
    \item $\kappa(\vec{W})$ the largest integer $k\geq 1$  such that $f$ attains its value at $\vec{W}$ with order~$k-1$.
\end{itemize}
Note that $\vec{W} \in\CRIT(f)$ if, and only if, $\kappa(\vec{W})\geq2$. Note also that, since $f$ is a polynomial of degree $2N$, if $\kappa(\vec{W})>2N$ then $f$ is constant. Hence, $\kappa(\vec{W})\leq 2N$ for every $\vec{W}$. The main goal of this section is the following statement.

\begin{proposition}
\label{prop_lnn_app}
Every saddle point $\vec{W}$ of $f$ such that $\zeta(\vec{W})=\kappa(\vec{W})$ is weakly strict.
\end{proposition}

\begin{corollary}
If $XY^T\neq0$ then $\vec{0}$ is a weakly strict saddle point.
\end{corollary}

\begin{proof}
The Taylor expansion of $f$ at $\vec{0}$ is
\[
f(\vec{W})= \frac{1}{2}\|Y\|^2 - \left< WX,Y \right> + \frac{1}{2}\|WX\|^2 = \frac{1}{2}\|Y\|^2 - P(\vec{W}) + O(\|\vec{W}\|^{2N})
\]
where $P(\vec{H})=\left< HX,Y \right>=\tr\, HXY^T$ is a non-zero homogeneous polynomial in $\vec{H}$ of degree $N$; here $\vec{H}=(H_1,\dots,H_N)$ and $H=H_N\cdots H_1$. It follows that  the critical point $\vec{0}$ is a saddle point and that $N=\zeta(\vec{0})=\kappa(\vec{0})$.
\end{proof}

We are now concerned with the proof of Proposition~\ref{prop_lnn_app}. We need two algebraic lemmas.

\begin{lemma}
\label{lemma_trace_zero}
Consider $V_1,\dots,V_N$ finite-dimensional real vector spaces, and set $V=V_1\times \dots\times V_N$. Denote the projection onto the $i$-th component by $\pi_i:V\to V_i$. Consider $1\leq k\leq N$, $1\leq i_1<\dots<i_k\leq N$ and $\vartheta:V_{i_1}\times\dots\times V_{i_k} \to \R$ a $k$-linear map. Set $P:V\to\R$, $P(v) = \vartheta(\pi_{i_1}v,\dots,\pi_{i_k}v)$. Denote by $\partial^2P(v):V\times V\to\R$ the second derivative of $P$ at $v\in V$, seen as a bilinear symmetric map. Then $\tr \, \partial^2P(v)=0$ holds for all $v\in V$.
\end{lemma}

\begin{proof}
Set $d=\sum_{i=1}^N\dim V_i$. Let $v\in V$ be arbitrary. The trace of $\partial^2P(v)$ is $\sum_{j=1}^d \partial^2P(v)(e_j,e_j)$ where $b=\{e_1,\dots,e_d\}$ is a basis of $V$. It does not depend on the choice of $b$. Choose $b$ consisting of vectors $e\in V$ with the following key property: $\pi_\lambda e\neq0 \Rightarrow \pi_\mu e=0$ for all $\mu\neq\lambda$. We compute
\[
\begin{aligned}
dP(v)u
&= \vartheta(\pi_{i_1}u,\pi_{i_2}v,\pi_{i_3}v,\dots,\pi_{i_k}v) \\
&\qquad + \vartheta(\pi_{i_1}v,\pi_{i_2}u,\pi_{i_3}v,\dots,\pi_{i_k}v) \\
&\qquad \quad \dots \\
&\qquad+ \vartheta(\pi_{i_1}v,\dots,\pi_{i_{k-1}}v,\pi_{i_k}u)
\end{aligned}
\]
and
\[
\begin{aligned}
&\partial^2P(v)(u_1,u_2) \\
&= \sum_{1\leq r<s\leq k} \vartheta(\pi_{i_1}v,\dots,\pi_{i_{r-1}}v,\pi_{i_{r}}u_1,\pi_{i_{r+1}}v,\dots,\pi_{i_{s-1}}v,\pi_{i_{s}}u_2,\pi_{i_{s+1}}v,\dots,\pi_{i_k}v) \\
&\qquad \qquad \quad + \vartheta(\pi_{i_1}v,\dots,\pi_{i_{r-1}}v,\pi_{i_{r}}u_2,\pi_{i_{r+1}}v,\dots,\pi_{i_{s-1}}v,\pi_{i_{s}}u_1,\pi_{i_{s+1}}v,\dots,\pi_{i_k}v)
\end{aligned}
\]
from where it follows that $\partial^2P(v)(e,e)=0$ for every element $e$ of a basis $b$ satisfying the key property mentioned above, in fact, if $e\in b$ then $\pi_{i_r}e\neq0 \Rightarrow \pi_{i_s}e=0$.
\end{proof}

\begin{lemma}
\label{lemma_spectrum_crucial}
Let $P:\R^d\to\R$ be a non-constant homogeneous polynomial of degree~$k$. Let $y_0\in\S^{d-1}$ be a critical point of $p=P|_{\S^{d-1}}$. If the matrix $\partial^2P(y_0)=[\partial^2_{ij}P(y_0)]$ satisfies $\tr\,\partial^2P(y_0) \leq k(k+d-2)P(y_0)$ and $D^2p(y_0)$ does not vanish as a bilinear form on $T_{y_0}\S^{d-1}$, then there exists $z\in T_{y_0}\S^{d-1}$ such that $D^2p(y_0)(z,z)<0$.
\end{lemma}

\begin{proof}
We write $\left<\cdot,\cdot\right>$ and $\|\cdot\|$ to denote the Euclidean inner-product and norm on~$\R^d$, respectively. Consider the function $h:\R^d\to\R$, $h(y)=\|y\|$, and its Euclidean gradient $V:\R^d\setminus\{0\} \to \R^d$, $V(y)=\frac{y}{\|y\|}$. Note that $y_0\in\S^{d-1}$ is a critical point of $p$ if, and only if, it is a critical point $P\circ V$, which means that $\nabla P(y_0)\in\R y_0$. The Hessian $D^2p(y_0):T_{y_0}\S^{d-1}\times T_{y_0}\S^{d-1}\to\R$ at such a critical point $y_0$ coincides with the restriction to $T_{y_0}\S^{d-1}=(\R y_0)^\bot$ of the Hessian $D^2\big(P\circ V\big)(y_0)$. In the following we will use $D^2$ also to denote the second derivative of maps seen as symmetric bilinear forms. Then $\nabla(P\circ V)(y_0)= DV(y_0)^T\nabla P(y_0)$ and
\[
\begin{aligned}
D^2(P\circ V)(y_0)(z_1,z_2) 
&= D^2P(y_0)(DV(y_0)z_1,DV(y_0) z_2) \\
&\qquad + \left< \nabla P(y_0),D^2V(y_0)(z_1,z_2)\right>.
\end{aligned}
\]
We compute $DV(y) z = \frac{1}{\|y\|} \left( z-\left<z,\frac{y}{\|y\|}\right> \frac{y}{\|y\|} \right) = h(y)^{-1} \big( z-\left<z,V(y)\right> V(y) \big)$ and
\[
\begin{aligned}
D^2V(y)(z_1,z_2)
&= \left<\nabla h^{-1}(y),z_2\right> \big( z_1-\left<z_1,V(y)\right> V(y) \big) \\
&\qquad - h(y)^{-1} \big( \left<z_1,DV(y)\cdot z_2\right> V(y) + \left<z_1,V(y)\right> DV(y)\cdot z_2  \big)
\end{aligned}
\]
for all $y\in\R^d\setminus\{0\}$ and $z_1,z_2\in\R^d$. Hence,
\begin{equation*}
D^2V(y)(z_1,z_2) = -\left<z_1,z_2\right>y \qquad \forall y\in\S^{d-1}, \ \forall z_1,z_2\in T_y\S^{d-1},
\end{equation*}
from where it follows that 
\begin{equation}
\begin{aligned}
D^2p(y_0)(z_1,z_2) 
&= D^2P(y_0)(z_1,z_2)-\left<z_1,z_2\right>\left<\nabla P(y_0),y_0\right> \\
&= D^2P(y_0)(z_1,z_2)-\left<z_1,z_2\right>kP(y_0)
\end{aligned}
\end{equation}
Let $e_2,\dots,e_d$ be an orthonormal basis of $T_{y_0}\S^{d-1}$ such that $D^2p(y_0)(e_\alpha,e_\beta)=0$ when $\alpha\neq \beta$. One can use the spectral theorem and the fact that $D^2p(y_0)$ is symmetric to obtain such a basis. Set $e_1=y_0$. Then
\begin{equation*}
\begin{aligned}
% \tr \, D^2p(y_0) &= 
\sum_{j=2}^d D^2p(y_0)(e_j,e_j)
&= \left(\sum_{j=2}^d D^2P(y_0)(e_j,e_j)\right) - (d-1)kP(y_0) \\
&= \left(\sum_{j=1}^d D^2P(y_0)(e_j,e_j)\right) - D^2P(y_0)(y_0,y_0) - (d-1)kP(y_0) \\
&= \tr \, \partial^2P(y_0) - (k-1)kP(y_0) - (d-1)kP(y_0) \\
&= \tr \, \partial^2P(y_0) - (k+d-2)kP(y_0) \leq 0.
\end{aligned}
\end{equation*}
If $D^2p(y_0)(e_j,e_j)$ vanishes for every $j=2,\dots,d$ then $D^2p(y_0)$ vanishes on $T_{y_0}\S^{d-1}$. Hence, $D^2p(y_0)(e_j,e_j)<0$ holds for some  $j\in\{2,\dots,d\}$.
\end{proof}

% \begin{remark}
When $V,E$ are vector spaces and $\theta:V^r\to E$ is a multi-linear map, its symmetrization is defined as 
\[
\theta^{\mathrm{sym}}(v_1,\dots,v_r) = \sum_{\sigma\in S_r} \theta^\sigma(v_1,\dots,v_r), \qquad \theta^\sigma(v_1,\dots,v_r) = \theta(v_{\sigma(1)},\dots,v_{\sigma(r)}).
\]
where $S_r$ denotes the group of permutations of $\{1,\dots,r\}$. This notation will be used in the proof below.
% \end{remark}

\begin{proof}
[Proof of Proposition~\ref{prop_lnn_app}]
Assume that $f$ is non-constant. Let $\vec{W}$ be a critical point of $f$ such that $k=\zeta(\vec{W})=\kappa(\vec{W})\leq N$. Let $I=(i_1,\dots,i_k) \in \{1,\dots,N\}^k$ be the indices $i$ such that $W_i=0$, labeled in increasing order $1\leq i_1<i_2<\dots<i_k\leq N$. Set $V_i=\R^{d_{i+1}\times d_{i}}$, so that $V = V_1\times\dots\times V_N$. The projection onto the $i$-th factor is denoted $\pi_i:V\to V_i$. Consider the multi-linear map $\vartheta : V_{i_1}\times\dots\times V_{i_k} \to\R$
\[
\vartheta(Q_1,\dots,Q_k) = \tr \ W_N\cdots W_{i_k+1}Q_{k}W_{i_k-1} \cdots W_{i_1+1}Q_{1}W_{i_1-1} \cdots W_1XY^T
% \begin{aligned}
% &\vartheta(Q_1,\dots,Q_k) 
% &= \tr \left( 
% \begin{aligned}
% &W_N\cdots W_{i_k+1}Q_{k}W_{i_k-1}\cdots W_{i_{k-1}+1}Q_{k-1}W_{i_{k-1}-1} \cdots \\
% &\cdots \\
% &\cdots W_{i_2+1}Q_{2}W_{i_2-1} \cdots W_{i_1+1}Q_{1}W_{i_1-1} \cdots W_1XY^T 
% \end{aligned}\right).
% \end{aligned}
\]
The $j$-th derivative $D^jf(\vec{W})$ vanishes for all $j=1,\dots,k-1$, and
\begin{equation}
\label{expression_k_th_derivative_of_f}
D^kf(\vec{W})(\vec{H^1},\dots,\vec{H^k}) = -\Theta^{\mathrm{sym}}(\vec{H^1},\dots,\vec{H^k})
\end{equation}
where
\begin{equation*}
% \label{formula_derivative_order_k_detailed}
\begin{aligned}
\Theta(\vec{H^1},\dots,\vec{H^k}) &= \vartheta(\pi_{i_1}(\vec{H^1}),\dots,\pi_{i_k}(\vec{H^k})) \\
% &\qquad = \tr \left( 
% \begin{aligned}
% &W_N\cdots W_{i_k+1}H^k_{i_k}W_{i_k-1}\cdots W_{i_{k-1}+1}H^{k-1}_{i_{k-1}}W_{i_{k-1}-1} \cdots \\
% &\cdots \\
% &\cdots W_{i_2+1}H^2_{i_2}W_{i_2-1} \cdots W_{i_1+1}H^1_{i_1}W_{i_1-1} \cdots W_1XY^T 
% \end{aligned}\right)
& = \tr \ W_N\cdots W_{i_k+1}H^k_{i_k}W_{i_k-1} \cdots W_{i_1+1}H^1_{i_1}W_{i_1-1} \cdots W_1XY^T
\end{aligned}
\end{equation*}
Consider the associated homogeneous polynomial $P(\vec{H})=D^kf(\vec{W})(\vec{H},\dots,\vec{H})$. It is non-constant and its degree is $k$. Note that
\[
P(\vec{H}) = -k! \ \Theta(\vec{H},\dots,\vec{H}) = -k! \ \vartheta(\pi_{i_1}(\vec{H}),\dots,\pi_{i_k}(\vec{H})).
\]
The above formula shows that $P$ fulfills the hypotheses of Lemma~\ref{lemma_trace_zero}. A direct application of this lemma shows that $\tr\ \partial^2P(\vec{H})=0$ for every $\vec{H}$. By Lemma~\ref{lemma_spectrum_crucial}, if $SV=\{\vec{H}\in V:\|\vec{H}\|=1\}$ and $\vec{H}_0\in SV$ is a critical point of $p=P|_{SV}$ satisfying $p(\vec{H}_0)\geq0$, then the Morse index of $\vec{H}_0$ is positive. In other words, $\vec{W}$ is a weakly strict saddle point of $f$.
\end{proof}

\section{Proof of Theorem~\ref{thm_weak_center_stable}}

The smooth Riemannian manifold $(M,g)$, the smooth function $f:M\to\R$ and the tamed weakly strict saddle point $x_0$ are fixed throughout this section. Assume that $f(x_0)=0$ without loss of generality. By Definition~\ref{defn_weak_saddle}, the jet of $f$ does not vanish to $\infty$ order at $x_0$, i.e. for some $k\geq2$, $f$ attains its value at $x_0$ with order $k-1$ but not with order $k$.

\subsection{Normal polar coordinates}
\label{ssec_normal_polar}

Let $\mathcal{O} \subset TM$ be the set consisting of tangent vectors $v$ such that the unique geodesic with initial velocity $v$ at time $t=0$ is defined up to time $t=1$. Then $\mathcal{O}$ is open and the exponential map is the smooth map defined as
\[
\exp:\mathcal{O}\to M, \qquad \exp(v) = \gamma(1)
\]
where $\gamma(t)$ is the geodesic satisfying $\dot\gamma(0)=v$. Note that $\mathcal{O}$ is a fiberwise star-shaped subset of $TM$. 

Given $\rho>0$, denote by $B_\rho$ the open $\rho$-ball $\{v\in T_{x_0}M:g(v,v)<\rho^2\}$. If $\rho$ is small enough, $\exp|_{B_\rho}:B_\rho\to U$ is a diffeomorphism onto an open neighborhood $U$ of $x_0$ in $M$. In fact, $U$ is the $\rho$-ball centered at $x_0$ with respect to the distance induced by $g$. A choice of $g$-orthonormal basis of $T_{x_0}M$ then determines a smooth coordinate system $w=(w_1,\dots,w_d)$ on $U$ via this diffeomorphism, where $w$ varies on the open Euclidean $\rho$-ball in $\R^d$. These coordinates are called normal coordinates at $x_0$. We might abuse notation and sometimes see $B_\rho$ as the open $\rho$-ball in the Euclidean $\R^d$. The metric $g$ gets represented as 
\begin{equation}
\label{g_in_normal_coords}
\sum_{i,j=1}^n \left(\delta_{ij} + r^2b_{ij}\right) dw_i\otimes dw_j
\end{equation}
for suitable smooth functions $b_{ij}$ satisfying $b_{ij}=b_{ji}$. Here $r^2=\sum_{i=1}^d w_i^2$. The Euclidean metric on $\R^d$ is denoted by 
\[
g_0
% =\left<\cdot,\cdot\right>
=\sum_{i=1}^d dw_i\otimes dw_i
\]
and the Euclidean norm by  $\|w\|=\sqrt{g_0(w,w)}=\sqrt{w_1^2+\dots+w_d^2}$. Then~\eqref{g_in_normal_coords} becomes
\begin{equation}
g = g_0 + r^2b, \qquad b = \sum_{i,j=1}^d b_{ij} \, dw_i\otimes dw_j.
\end{equation}
Let $\S^{d-1} = \{ w \in \R^d \colon \|w\|=1 \}$ be the unit sphere and $\iota:\S^{d-1} \to \R^d$ be the inclusion map. Then $h = \iota^*g_0$ will be referred to as the Euclidean metric on~$\S^{d-1}$. 

Let $\phi$ be the ``polar coordinates map'' 
\[
\phi:\R\times \S^{d-1} \to \R^d, \qquad \phi(r,u)=ru.
\]
It defines a $C^\infty$-diffeomorphism $(0,+\infty)\times \S^{d-1} \simeq \R^d \setminus \{0\}$. The differential $d\phi(r,u) : \R\times T_u\S^{d-1} \to T_{ru}\R^d = \R^d$ is the linear map 
\[
(\delta r,\delta u) \mapsto \delta r \, u + r \, \delta u
\]
as  shown by a direct calculation. We always identify 
\[
T_{(r,u)}\R\times \S^{d-1} = T_r\R\times T_u\S^{d-1} = \R\times T_u\S^{d-1}
\]
so that vector fields on some open set $\mathcal{U}\subset\R\times \S^{d-1}$ can be seen as functions
\[
V:\mathcal{U} \to \R\times T\S^{d-1}, \qquad V(r,u) = (V_1(r,u),V_2(r,u)) \in \R\times T_u\S^{d-1},
\]
in other words, functions $V:\mathcal{U}\to\R\times \R^d$ satisfying $g_0(V_2(r,u),u)=0$. 
% The partial derivatives with respect to the $r$-coordinate and the $u$-coordinate will be denoted by $D_1V$ and $D_2V$, respectively.

On $\R\times \S^{d-1}$ consider the Riemannian metric $g_{\mathrm{ref}} = dr \otimes dr + h$ where we write $h$ to denote the pull-back of $h$ by the projection $\R\times \S^{d-1}\to \S^{d-1}$ with no fear of ambiguity. On $(-\rho,\rho)\times \S^{d-1}$ we also have two smooth $(0,2)$-tensors $\phi^*g$, $\phi^*g_0$. They define Riemannian metrics on $(0,\rho)\times \S^{d-1}$ that become degenerate at $r=0$.  We would like to write $\phi^*g$ with respect to $g_{\mathrm{ref}}$ as 
\begin{equation}
\label{g_in_terms_of_g_ref}
\phi^*g(\cdot,\cdot)=g_{\mathrm{ref}}(\Lambda\cdot,\cdot)
\end{equation}
where $\Lambda$ is the smooth field on $(-\rho,\rho)\times \S^{d-1}$ of $g_{\mathrm{ref}}$-symmetric linear maps
\[
\Lambda(r,u) : T_{(r,u)}(\R\times \S^{d-1})\to T_{(r,u)}(\R\times \S^{d-1})
\]
determined by~\eqref{g_in_terms_of_g_ref}. Consider the smooth matrix-valued function $B=[b_{ij}]$ defined on the $\rho$-ball. It satisfies $B^T=B$ pointwise. For each $u\in \S^{d-1}$, $\pi_uw = w-g_0(w,u)u$ denotes the $g_0$-orthogonal projection onto $T_u\S^{d-1}$. Let smooth functions $\lambda$, $v$ and~$\alpha$ be defined on $(-\rho,\rho)\times \S^{d-1}$ by
\begin{equation}
\label{fncts_lambda_v_alpha}
\begin{aligned}
\lambda(r,u) &= b(ru)(u,u) =g_0(B(ru)u,u)\\
v(r,u) &= \pi_uB(ru)u\\
\alpha(r,u)&=\pi_uB(ru)
\end{aligned}
% \lambda(r,u) = b(ru)(u,u) =g_0(B(ru)u,u), \ v(r,u) = P_uB(ru)u, \ \alpha(r,u)=P_uB(ru)
\end{equation}
respectively. Note that $\lambda(r,u)\in\R$, $v(r,u)\in T_u\S^{d-1}$, $\alpha(r,u):T_u\S^{d-1}\to T_u\S^{d-1}$ is an $h$-symmetric linear map. A direct calculation shows that
\begin{equation}
\Lambda = 
\begin{pmatrix}
1+r^2\lambda & r^3h(v,\cdot) \\ r^3v & r^2I+r^4\alpha
\end{pmatrix}
\end{equation}
where $\Lambda$ is written in blocks with respect to the splitting $\R\times T_u\S^{d-1}$. If $|r|>0$ is small enough then $\Lambda(r,u)$ is invertible, and one readily computes its inverse to be
\begin{equation}
\label{Lambda_inverse}
\Lambda^{-1} =
\begin{pmatrix}
c & -rc\,h(w,\cdot) \\
-rcw & r^{-2}(I+r^2\alpha)^{-1} + r^2c\,h(w,\cdot)w
\end{pmatrix}
\end{equation}
where
\begin{equation}
\label{fncts_w_c}
\begin{aligned}
w(r,u) &= (I+r^2\alpha)^{-1}v  \\
c(r,u) &= \frac{1}{1+r^2\lambda-r^4\,h(v,(I+r^2\alpha)^{-1}v)}
\end{aligned}
\end{equation}
It is important to note that $w(r,u)\in T_u\S^{d-1}$ and $c(r,u)\in\R$ are smooth functions defined for all $(r,u) \in(-\tau,\tau)\times \S^{d-1}$ provided $\tau>0$ is small enough, but $\Lambda(r,u)^{-1}$ is not defined for $r=0$. 

We use the above to compute the gradient with respect to $\phi^*g$ of a $C^1$ function $F$ defined on an open subset of $(0,\rho)\times \S^{d-1}$. This is the vector field $\nabla^{\phi^*g}F=(a,z)$ determined by
\[
g_{\mathrm{ref}}((D_1F,\nabla F^r),\cdot)=dF = \phi^*g(\nabla^{\phi^*g}F,\cdot) = g_{\mathrm{ref}}(\Lambda \nabla^{\phi^*g}F,\cdot)
\]
where
\[
D_1F(r,u) = \lim_{t\to0} \frac{F(r+t,u)-F(r,u)}{t},
\]
$F^r:\S^{d-1}\to\R$ is the function $F^r=F(r,\cdot)$, and $\nabla F^r$ is its $h$-gradient. Hence,
\begin{equation*}
\begin{aligned}
\begin{pmatrix} a \\ z \end{pmatrix}
&= \Lambda^{-1} \begin{pmatrix} D_1F \\ \nabla F^r \end{pmatrix} 
\end{aligned} 
\end{equation*}
from where we get
\begin{equation}
\label{main_formula_gradient}
\begin{aligned} 
a(r,u) &= c\, \big(D_1F - r\,h(w,\nabla F^r)\big) \\ 
z(r,u) &= -c\, r \big(  D_1F - r\,h(w,\nabla F^r)\big)w + r^{-2}(I+r^2\alpha)^{-1}\nabla F^r %\qquad -rc\,D_1F\,w + r^{-2}(I+r^2\alpha)^{-1}\nabla F^r + r^2c\,h(w,\nabla F^r)w 
\end{aligned}
\end{equation}
in view of~\eqref{Lambda_inverse}.

\subsection{The blow-up construction}
\label{ssec_blow_up}

Let $w=(w_1,\dots,w_d)$ be normal coordinates defined on the image~$U$  by the exponential map of an open ball in $T_{x_0}M$ centered at the origin of small radius $\rho>0$, as in subsection~\ref{ssec_normal_polar}. In these coordinates $x_0\simeq 0$. Write $f=f(w)$ on $U$, with no fear of ambiguity. There is a non-zero homogeneous polynomial $P:\R^d\to\R$ of degree $k$ such that
\begin{equation}
f(w) = P(w) + \varepsilon(w)
\end{equation}
where $\varepsilon$ is smooth and satisfies 
\begin{equation}
\label{estimate_varepsilon}
\text{$\varepsilon(w)= O(\|w\|^{k+1})$ as $\|w\|\to 0^+$.}
\end{equation}
Define for $(r,u) \in (-\rho,\rho) \times \S^{d-1}$
\begin{equation*}
F(r,u) := f(ru) = r^kp(u) + \varepsilon(ru)
\end{equation*}
where $p:\S^{d-1} \to \R$ is defined by $p = P|_{\S^{d-1}}$. Note that $F(r,u)$ is a smooth function $(-\rho,\rho) \times \S^{d-1} \to \R$. Hence, we can find smooth functions
\begin{equation}
\begin{aligned}
& \theta:(-\rho,\rho) \times \S^{d-1}\to\R \\
& Z:(-\rho,\rho) \times \S^{d-1}\to T\S^{d-1}
\end{aligned}
\end{equation}
satisfying $Z(r,u)\in T_u\S^{d-1}$ and
\begin{equation}
\label{exp_grad_at_r=0}
\begin{aligned}
D_1F(r,u) &= r^{k-1}kp(u) +r^k\theta(r,u) \\
\nabla F^r(u) &= r^k\nabla p(u)+r^{k+1}Z(r,u)
\end{aligned}
\end{equation}
where $\nabla p$ is the $h$-gradient of $p$. This shows that $\nabla^{\phi^*g}F$, which in principle was well-defined and smooth only on $(0,\rho)\times
S^{d-1}$, extends smoothly to $(-\rho,\rho)\times\S^{d-1}$. But, in fact, more is true: plugging~\eqref{exp_grad_at_r=0} into~\eqref{main_formula_gradient} we obtain
\begin{equation}
\begin{aligned}
r^{2-k}a(r,u) &= rkp(u)+r^2\theta(r,u) + r^3\hat{\theta}(r,u) \\
r^{2-k}z(r,u) &= \nabla p(u)+rZ(r,u) + r^2\hat{Z}(r,u)
\end{aligned}
\end{equation}
for suitable smooth functions $\hat{\theta},\hat{Z}$ defined on $(-\tau,\tau)\times\S^{d-1}$ for some $0<\tau<\rho$ small. Here~\eqref{fncts_lambda_v_alpha} and~\eqref{fncts_w_c} are strongly used. The following lemma follows immediately.

\begin{lemma}
\label{lemma_C1_extension}
\label{lemma_extension_of_gradient}
The vector field $r^{2-k}\nabla^{\phi^*g}F$ admits a smooth extension to a vector field~$X$ on $(-\rho,\rho) \times \S^{d-1}$ that satisfies $X(0,u)=(0,\nabla p(u))$. If $\sigma\subset\R$ is the spectrum of the Hessian of $p$ at some $u_*\in\CRIT(p)$, then the spectrum of the linearization of $X$ at the corresponding rest point $(0,u_*) \in \{0\}\times \S^{d-1}$ is precisely $\{kp(u_*)\}\cup\sigma$.
\end{lemma}

\subsection{Controlling trajectories}

\begin{definition}
\label{defn_omega_limit_sets}
The following notions of~$\omega$-limit sets will be useful:
\begin{itemize}
\item Let $X$ be a topological space and ~$c:[0,+\infty) \to X$ be continuous. Its~$\omega$-limit set is defined as $$ \omega(c) = \{ x \in X : \exists t_n \to+\infty \ \text{such that} \ c(t_n) \to x \} \, . $$
\item Let~$w_0 \in M$ and $c:[0,+\infty) \to M \setminus\{w_0\}$ be continuous. Assume that $w_0 \in \omega(c)$. The secondary~$\omega$-limit set
\[
\omega_2(c;w_0) \subset S_{w_0}M
\]
is defined as the set of $u_*$ for which there exist sequences $t_n\to+\infty$ and $u_n\in S_{w_0}M$ satisfying $$ r_n=\DIST(c(t_n),w_0)\to0, \qquad c(t_n)=\exp(r_nu_n), \qquad u_n\to u_*. $$
\end{itemize}
\end{definition}

\begin{lemma}
\label{rmk_omega_sets}
The following facts hold:
\begin{itemize}
\item $\omega$-limit sets are always closed. If $c$ has pre-compact image then $\omega(c)$ is compact connected and non-empty.
\item Assume $X$ is a locally compact metric space. Then:
\begin{itemize}
    \item If $\omega(c)$ is compact non-empty, then the image of $c$ is pre-compact.
    \item If $\omega(c)$ is disconnected then all its connected components are not compact.
\end{itemize}
\item If $w:[0,+\infty) \to M$ is an integral curve of~$-\nabla f$ then~$\omega(w)\subset\CRIT(f)$. If~$\omega(w)$ contains two distinct points then~$w$ has infinite length.
\end{itemize}
\end{lemma}

\begin{proof}
It is clear from $\omega(c) = \cap_{n\in\N} cl\big(c([n,+\infty))\big)$ that $\omega(c)$ is closed. Here $cl$ denotes closure in $X$. If $cl\big(c([0,+\infty))\big)$ is compact then $\omega(c)$ is the intersection of a nested sequence of compact connected non-empty subsets and, as such, it is also compact connected and non-empty.

% Let $cl\big(c([0,+\infty))\big)$ be compact, and assume $X$ is a locally compact metric space with the metric denoted by $d$. $\omega(c)$ is non-empty and compact. 

% , since it is an intersection of nested compact sets. Also, if $cl\big(c([0,+\infty)])\big)$ is compact then, since closures of connected subsets are connected, $cl\big(c([n,+\infty))\big)$ is a sequence of nested compact connected subsets and the

% and connected, that $c$ has infinite length when $\omega(c)$ has more than one point, and that $\omega(w)$ consists of critical points when $w(t)$ is an gradient trajectory. The proofs of these facts are left to the reader.

From now on assume that $(X,d)$ is a locally compact metric space. 

Suppose that $\omega(c)$ is compact non-empty. Then $K=\{x\in X:d(x,\omega(c))\leq\delta\}$ is compact provided $\delta>0$ is small enough. If there is no $T\geq0$ such that $c([T,+\infty)) \subset K$, then one finds $t'_n\to+\infty$ such that $d(c(t'_n),\omega(c))>\delta$. By definition of $\omega(c)$ we find $t_n\to+\infty$ such that $d(c(t_n),\omega(c))<\delta$. By the intermediate value theorem there exist $s_n\to+\infty$ satisfying $d(c(s_n),\omega(c))=\delta$. By compactness of $K$ we can assume, without loss of generality, that $c(s_n)\to w$, $d(w,\omega(c))=\delta$. This contradicts $w\in\omega(c)$. Hence, $c([T,+\infty)) \subset K$ for some $T\geq0$.

Assume now that $\omega(c)$ is disconnected. Hence, by the above, the image of $c$ is not pre-compact. We wish to show that an arbitrary connected component of $\omega(c)$ is not compact. To this end, it suffices to show that an arbitrary compact connected set $\omega_*\subset \omega(c)$ is contained in a strictly larger connected subset of $\omega(c)$. Choose $\delta>0$ small enough so that $K=\{w\in M:\DIST(w,\omega_*)\leq\delta\}$ is compact. Define a set $Q$ by requiring that $q\in Q$ if, and only if, there exist sequences $t_n\leq s_n$ such that: $t_n\to+\infty$, $\DIST(c(t_n),\omega_*) \to 0$, $c([t_n,s_n])\subset K$, $c(s_n)\to q$. It follows that $\omega_*\subset Q\subset K$. First we claim that $Q$ is closed. Let $q_k\in Q$, $q_k\to q\in K$, and consider corresponding sequences $t^k_n\leq s^k_n$ as above. For each $k$, there exists $n_k$ such that if $n\geq n_k$ then $t^k_n\geq k$, $\DIST(c(t^k_n),\omega_*)\leq \frac{1}{k}$, $d(c(s^k_n),q_k)\leq\frac{1}{k}$. Hence, taking $t'_k=t^k_{n_k}$, $s'_k=s^k_{n_k}$ we obtain $\DIST(c(t'_k),\omega_*) \to0$, $c(s'_k)\to q$ and $c([t'_k,s'_k])\subset K$. Thus, $q\in Q$. Having proved that $Q$ is closed, we conclude that $Q$ is compact since so is $K$. Now we show that $Q$ is connected. Consider $U,V\subset M$ open sets such that $Q\subset U\cup V$ and $Q\cap U\cap V=\emptyset$. Both $Q\cap U$ and $Q\cap V$ are closed in $Q$, hence compact. Hence, up to shrinking $U$ and $V$, it can be assumed that $U\cap V=\emptyset$. Since $\omega_*$ is connected we can assume, up to relabeling, that $\omega_*\subset U$. Assume, by contradiction, that $Q\cap V \neq\emptyset$. Choose $q\in Q\cap V$. By definition, we find $t_n\leq s_n$ such that $t_n\to+\infty$, $\DIST(c(t_n),\omega_*) \to 0$, $c([t_n,s_n])\subset K$, $c(s_n)\to q$. Hence, there exists $s'_n \in (t_n,s_n)$ such that $c(s'_n)\in K \setminus (U\cup V)$. There exists a subsequence $s'_{n_j}$ such that $c(s'_{n_j})\to q'\in K\setminus (U\cup V)$ as $j\to\infty$. Note that $\DIST(c(t_{n_j}),\omega_*)\to0$, $c([t_{n_j},s_{n_j}'])\subset c([t_{n_j},s_{n_j}])\subset K$ as $j\to\infty$. Hence, $q'\in Q\setminus (U\cup V)$ which is an absurd. Hence, $Q\subset U$. This proves that $Q$ is connected. The final task is now to show that $Q\setminus \omega_*\neq\emptyset$. Since the image of $c$ is not pre-compact and $\omega_*\subset\omega(c)$, we find $t_n <\tau_n$ such that, $t_n\to+\infty$, $\DIST(c(t_n),\omega_*)\to0$ and $\DIST(c(\tau_n),\omega_*) >\delta$. By the intermediate value theorem, we find $s_n \in (t_n,\tau_n)$ such that $c([t_n,s_n])\subset K$ and $\DIST(c(s_n),\omega_*)=\frac{\delta}{2}$. Up to choice of a subsequence, $c(s_n)$ converges to a point  $q\in Q$ satisfying $\DIST(q,\omega_*)=\frac{\delta}{2}$. In particular, $q\in Q\setminus \omega_*$. It has been proved that $Q$ is a connected subset of $\omega(c)$ such that $\omega_*\subset Q$ and $Q\setminus\omega_*\neq\emptyset$.

It remains to be shown that if $w:[0,+\infty) \to M$ is an integral curve of~$-\nabla f$ then~$\omega(w)\subset\CRIT(f)$.
% , and that if $\omega(c)$ contains two distinct points then~$c$ has infinite length. These facts are 
This is standard and its proof will not be presented here.
\end{proof}

Recall the radius $\rho>0$ of the ball in $T_{x_0}M$ centered at the origin, fixed in subsection~\ref{ssec_normal_polar}.

\begin{lemma}
\label{lemma_simple_analysis}
For every $\delta>0$ there exists $0<R<\rho$ such that if $$ \gamma:[0,+\infty) \mapsto (0,R]\times \S^{d-1} \qquad \gamma(t) = (r(t),u(t)) $$ is an integral curve of $-X$ then $$ \begin{aligned} & \sup \ \{ \DIST(x,\CRIT(p)) : x\in \omega(u(t)) \} \leq \delta, \\ & \DIAM \ p\big(\omega(u(t))\big) = \sup \ \{ |p(x)-p(y)| : x,y \in \omega(u(t)) \} \leq \delta . \end{aligned} $$
\end{lemma}

\begin{proof}
In the following we might use $\left<\cdot,\cdot\right>$ to denote the Riemannian metric $h=j^*g_0$ on $\S^{d-1}$ or the metric $g_0$ on $\R^d$, and $\|u\|=\|u\|_h$, $\|v\|=\|v\|_{g_0}$ to denote the corresponding norms of vectors $u\in T\S^{d-1}$, $v\in\R^d$. We might write $\|q\|_{L^\infty}$ instead of $\|q\|_{L^\infty([0,+\infty))}$ to denote the $L^\infty$-norm of a function $q:[0,+\infty)\to\R$.

The first step is to address the claim about $\DIAM \ p\big(\omega(u(t))\big)$. To this end consider $a<b$ regular values of $p$. We claim that there exists $L = L(a,b)>0$ such that if $u:[0,+\infty) \to \S^{d-1}$ is $C^1$ and satisfies 
\[
\|\dot u + \nabla p(u)\|_{L^\infty} < L, \qquad p(\omega(u(t))) \cap [a,b] \neq \emptyset 
\]
then
\[
\omega(u(t)) \subset p^{-1}((a,b)) 
\]
holds. In particular, $\DIAM \ p\big(\omega(u(t))\big) < b-a$. To prove the claim, note first that $K = p^{-1}([a,b])$ is a smooth compact manifold with boundary $\partial K = p^{-1}(a) \cup p^{-1}(b)$. On $p^{-1}(b)$ the vector field $-\nabla p$ points into $K$ transversely to the boundary, on $p^{-1}(a)$ the vector field $-\nabla p$ points out of $K$ transversely to the boundary. By compactness of $K$, there exists $L(a,b)>0$ such that if $\|\dot u+\nabla p(u)\|_{L^\infty}<L(a,b)$ then $\dot u(t)$ points into $K$ transversely to the boundary for all $t$ such that $p(u(t))=b$, and $\dot u(t)$ points out of $K$ transversely to the boundary for all $t$ such that $p(u(t))=a$. It follows that $u^{-1}(K)$ is an interval if $u^{-1}(K)\neq\emptyset$ and $\|\dot u+\nabla p(u)\|_{L^\infty}<L(a,b)$. In this case, if $\omega(u(t)) \cap K \neq \emptyset$ then the interval $u^{-1}(K)$ is not bounded from above and $u(t)$ stays at a positive distance from $\partial K$ when~$t$ is large enough. Hence, $\omega(u(t)) \subset K \setminus \partial K$ in this case. Note that if $p\big(\omega(u(t))\big) \cap [a,b] \neq \emptyset$ then $\omega(u(t)) \cap K \neq \emptyset$ and the above conclusions hold, that is, $p\big(\omega(u(t))\big) \subset (a,b)$ and $\omega(u(t)) \subset K \setminus \partial K$.

Now let $\delta>0$ be given arbitrarily. Consider regular values of $p$ $$ a_0 < \min p < a_1 < \dots < a_{N-1} < \max p < a_N $$ such that $a_{i+1}-a_i < \delta$ for all $0\leq i\leq N-1$. Choose $L$ such that 
\[
0 < L < \min_i \ L(a_i,a_{i+1})
\] holds, and suppose that $u:[0,+\infty) \to \S^{d-1}$ satisfies $\|\dot u + \nabla p(u)\|_{L^\infty} < L$.
Note that $p\big(\omega(u(t))\big)$ is a closed interval, since $\omega(u(t))$ is a connected compact subset of $\S^{d-1}$. We find $i$ such that $p(\omega(u(t))) \cap [a_i,a_{i+1}] \neq \emptyset$. Above it was shown that $p(\omega(u(t))) \subset (a_i,a_{i+1})$. Hence $\DIAM \ p\big(\omega(u(t))\big) < \delta$.

By Lemma~\ref{lemma_C1_extension} there exist $0<R_0<\frac{\rho}{2}$, $C>0$ such that
\begin{equation}
\label{X_2}
0<r\leq R_0 \ \Rightarrow \ |X_1(r,u)| + \|X_2(r,u) - \nabla p(u)\| \leq C r
\end{equation}
where we write $X=(X_1,X_2)$ according to $T_{(r,u)}\R\times \S^{d-1} = \R \times T_u\S^{d-1}$. Hence, a trajectory $\gamma(t)=(r(t),u(t))$ ($t\geq 0$) of $-X$ contained in $(0,R_0]\times \S^{d-1}$ necessarily satisfies $\|\dot u(t)-\nabla p(u(t))\|_{L^\infty} \leq C\|r\|_{L^\infty}$. Combining with what was proved before, one gets that $\DIAM \ p\big(\omega(u(t))\big)$ becomes arbitrarily small provided $\|r\|_{L^\infty}$ is small enough. 

For the remaining of the proof we argue by contradiction and assume that for some $\delta>0$ there exists a sequence $$ \gamma_n(t) = (r_n(t),u_n(t)):[0,+\infty) \to (0,\rho) \times \S^{d-1} \qquad (n\in \N) $$ of integral curves of $-X$ satisfying 
\begin{equation}
\label{curves_gamma_n_conditions}
\lim_{n\to +\infty} \ \|r_n\|_{L^\infty} = 0, \quad
\forall n \ \exists x_n \in \omega(u_n(t)): \ \DIST(x_n,\CRIT(p)) \geq \delta.
\end{equation}
Let $n_0 \in \N$ satisfy $n\geq n_0 \Rightarrow \|r_n\|_{L^\infty} \leq R_0$. By~\eqref{X_2} 
\begin{equation}
\label{esimate_from_radius}
\begin{aligned} 
n\geq n_0 \ \Rightarrow \ & |\dot r_n|+\|\dot u_n + \nabla p(u_n) \| _{L^\infty} \\
&\qquad= |X_1(r_n,u_n)|+\| \nabla p(u_n)-X_2(\gamma_n) \|_{L^\infty} \leq C\|r_n\|_{L^\infty}.
\end{aligned}
\end{equation}
By the first part of the proof
\begin{equation}
\label{diam_goes_to_zero}
\lim_{n\to+\infty} \ \DIAM \ p\big(\omega(u_n(t))\big) = 0.
\end{equation}
Up to choice of subsequence, assume $x_n \to x_*$. Then $\DIST(x_*,\CRIT(p)) \geq \delta$. Denote
$$ \eta = \|\nabla p(x_*)\| > 0, \qquad M = \sup \ \{|p(u)|+\|\nabla p(u)\| : u \in \S^{d-1} \}. $$
Recall the Riemannian metric $g_{\mathrm{ref}} = dr \otimes dr + h$ on $\R \times \S^{d-1}$. Note that we have $\|X(0,x_*)\|_{g_{\mathrm{ref}}} = \|\nabla p(x_*)\| = \eta$. Let $R_1>0$ and $U \subset \S^{d-1}$ be such that $2R_1<R_0$, $U$ is an open neighborhood of $x_*$ and
\begin{equation}
\label{size_X_g_1}
\forall (r,u) \in [0,R_1) \times U: \qquad \|X(r,u)\|_{g_{\mathrm{ref}}} \geq \frac{\eta}{2} \, .
\end{equation}
Let $d_{\mathrm{ref}}$ be the distance function on $\R\times \S^{d-1}$ associated to~$g_{\mathrm{ref}}$ and consider a closed neighborhood~$W \subset U$ of~$x_*$ in~$\S^{d-1}$ such that the $d_{\mathrm{ref}}$-distance from $[0,R_1/2]\times W$ to the complement of $[0,R_1) \times U$ in $[0,+\infty) \times \S^{d-1}$ is a positive number $\rho_1$. Choose 
\[
M' > \sup \ \{\|X(r,u)\|_{g_{\mathrm{ref}}} : (r,u) \in [0,R_0]\times \S^{d-1} \}.
\]
If $\alpha:[a,b] \to [0,+\infty) \times \S^{d-1}$ is $C^1$, $\alpha(a) \in [0,R_1/2]\times W$, $\alpha(b) \not\in [0,R_1) \times U$ and $\|\dot\alpha(t)\|_{g_{\mathrm{ref}}} \leq m$ holds for all $t\in [a,b]$, then
\begin{equation}
\label{est_b-a}
\rho_1 \leq \int_a^b \|\dot\alpha(t)\|_{g_{\mathrm{ref}}} dt \leq (b-a)m \quad \Rightarrow \quad b-a \geq \frac{\rho_1}{m} \, .
\end{equation}
Set $\tau = \frac{\rho_1}{2M'}$. We claim that for all $n\in\N$ and $t_* \geq \tau$ such that $\gamma_n(t_*) \in [0,R_1/2] \times W$ the following holds:
$$ 
\gamma_n([t_*-\tau,t_*+\tau]) \subset [0,R_1) \times U. 
%\Rightarrow \ \forall t\in [t_*-\tau,t_*+\tau]: \ \|X(\gamma_n(t))\|_{g_1} \geq \eta/2.
$$
For if not, then there would exist $t'\in [t_*-\tau,t_*+\tau]$ such that $\gamma_n(t')\not\in[0,R_1)\times U$, and the closed interval $J\subset[t_*-\tau,t_*+\tau]$ bounded by $t'$ and $t_*$ would satisfy $\gamma_n(J) \subset [0,R_0] \times \S^{d-1}$. Since $\|\dot\gamma_n(t)\|_{g_{\mathrm{ref}}} < M'$ holds for all $t\in J$ in view of the definition of $M'$, estimate~\eqref{est_b-a} implies $\tau\geq |t_*-t'| \geq\frac{\rho_1}{M'}$, in contradiction to the definition of $\tau$. Hence, 
\[
\forall t\in[t_*-\tau,t_*+\tau]: \qquad \|X(\gamma_n(t))\|_{g_{\mathrm{ref}}} \geq \frac{\eta}{2}. 
\]
holds for all $n$ and $t_*\geq\tau$ satisfying $\gamma_n(t_*) \in[0,R_1/2]\times W$. 

Let $n_1$ be such that $n\geq n_1$ implies $C\|r_n\|_{L^\infty} \leq \eta/4$. If $n\geq\max\{n_0,n_1\}$ and $t_* \geq \tau$ is such that $\gamma_n(t_*) \in [0,R_1/2] \times W$, then for all $t\in [t_*-\tau,t_*+\tau]$ the following estimate holds:
$$
\begin{aligned}
\|\nabla p(u_n(t))\| 
& \geq \|\dot u_n(t)\| - \|\nabla p(u_n(t))+\dot u_n(t)\|  \\
& = |\dot r_n(t)| + \|\dot u_n(t)\| - |\dot r_n(t)| - \|\nabla p(u_n(t))+\dot u_n(t)\| \\
& \geq \sqrt{|\dot r_n(t)|^2 + \|\dot u_n(t)\|^2} - C\|r_n\|_{L^\infty} \\
& = \|X(\gamma_n(t))\|_{g_{\mathrm{ref}}} - C\|r_n\|_{L^\infty} \geq \frac{\eta}{2}-\frac{\eta}{4} = \frac{\eta}{4}.
\end{aligned}
$$
Note that
\begin{equation}
\label{p(u(t))_decay}
\begin{aligned}
&\frac{d}{dt}p(u_n(t)) = dp(u_n(t)) \, \dot u_n(t) \\ 
&= -\|\nabla p(u_n(t))\|^2 + \left< \nabla p(u_n(t)),\dot u_n(t) +\nabla p(u_n(t)) \right> \\ 
&\leq -\|\nabla p(u_n(t))\|^2 + MC\|r_n\|_{L^\infty([0,+\infty)])}
\end{aligned}
\end{equation}
holds provided $n\geq n_0$.
Let $n_2$ satisfy $n\geq n_2 \Rightarrow MC\|r_n\|_{L^\infty([0,+\infty)])} \leq \eta^2/32$.
Then, by~\eqref{p(u(t))_decay}, for all $n\geq \max\{n_0,n_1,n_2\}$ and $t_* \geq \tau \ \text{such that} \ \gamma_n(t_*) \in [0,R_1/2] \times W $, we have
$$
\begin{aligned}
p(u_n(t_*+\tau)) \leq p(u_n(t_*)) + \tau(-\eta^2/16 + MC\|r_n\|_{L^\infty}) \leq p(u_n(t_*)) -\tau\frac{\eta^2}{32} \, .
\end{aligned}
% \end{aligned}
$$
Since $x_n \to x_*$, $x_n\in\omega(u_n(t))$, $x_*$ is an interior point of $W$, $\|r_n\|_{L^\infty}\to0$, we find that for each $n$ large enough then there exists sequence $t^j_{*} \to +\infty$ ($j\to+\infty$), such that $\gamma_n(t^j_*)\in [0,R_1/2]\times W$. The above gives $p(u_n(t^j_*+\tau)) \leq p(u_n(t^j_*)) - \tau\eta^2/32$. With such $n$ fixed, up to choice of subsequence, we find $x',x'' \in \omega(u_n)$ such that $u_n(t^j_*+\tau)\to x'$, $u_n(t^j_*)\to x''$ as $j\to+\infty$. By the continuity of $p$, $p(x'')-p(x')\geq \tau\eta^2/32$. It follows that $p\big(\omega(u_n(t))\big)$ has diameter larger than $\tau\frac{\eta^2}{32}>0$ provided $n$ is large enough. This is in contradiction to~\eqref{diam_goes_to_zero}.
\end{proof}

The following lemma is well-known for real-analytic functions on Euclidean space. We include a proof for the sake of completeness since we need it for general real-analytic Riemannian manifolds.

\begin{lemma}
\label{lemma_omega_limit_set_one_pt}
Let~$w:[0,+\infty) \to M$ be a non-constant solution of $\dot w = -\nabla f(w)$. Assume that~$(M,g)$ and~$f$ are real-analytic, and that~$\omega(w)$ is compact and non-empty. Then there exists a critical point $y_0$ of~$f$ such that $\lim_{t\to+\infty} w(t) = y_0$.
\end{lemma}

\begin{proof}
By Remark~\ref{rmk_omega_sets} and the assumptions, $\omega(w)$ is a connected compact non-empty set consisting of critical points of~$f$.
Let $c \in \R$ be such that $\{c\}= f(\omega(w))$, i.e. $c$ is the critical value attained along~$\omega(w)$.
Note that~$f(w(t)) > c$ for all $t\geq 0$ since~$\lim_{t\to+\infty} f(w(t)) = c$ and~$f(w(t))$ is strictly decreasing.

The main ingredient in this proof is Lojasiewicz's inequality: for every $w_0 \in M$ there exist constants $0 < \vartheta < 1$, $C > 0$ such that $|f(w)-f(w_0)|^\vartheta \leq C \|\nabla f(w)\|$ for all $w$ on some neighborhood of~$w_0$. Here we write $\|\cdot\|$ instead of $\|\cdot\|_g$ for simplicity. By compactness of~$\omega(w)$ and a simple covering argument, one finds an open neighborhood~$U_0$ of~$\omega(w)$ and constants $0 < \vartheta < 1$, $C > 0$ such that
\begin{equation}
\label{Lojasiewicz}
|f(w)-c|^\vartheta \leq C\|\nabla f(w)\| \qquad  \forall w\in U_0.
\end{equation}
Let~$U_1$ be an open neighborhood of~$\omega(w)$ such that $|f(w)-c|\leq 1$ for all~$w\in U_1$.
Let $t_0 > 0$ be such that $t\geq t_0 \Rightarrow w(t) \in U := U_0 \cap U_1$.
On~$U$ we can replace~$\vartheta$ in inequality~\eqref{Lojasiewicz} by larger exponents smaller than~$1$.
Set $$ h(t) = f(w(t))-c $$ so that $h(t)>0$ and $h'(t) = -\|\nabla f(w(t))\|^2$. 
One finds $\frac{1}{2} < \beta < 1$ such that
\begin{equation}
\label{diff_ineq_h}
\forall t\geq t_0: \qquad\qquad h(t)^\beta \leq C \sqrt{-h'(t)} \, . %\qquad \qquad \frac{1}{2} < \beta < 1.
\end{equation}
Set $$ \epsilon_* = \frac{2(1-\beta)}{2\beta-1} > 0. $$
We claim that
\begin{equation}
\label{decay_h}
\forall 0<\epsilon<\epsilon_*: \qquad  \qquad \lim_{t\to+\infty} t^{1+\epsilon} h(t) = 0.
\end{equation}
We will now prove this claim.
Since $-1 < 1-2\beta <0$, we get from~\eqref{diff_ineq_h} $$ \forall t\geq t_0: \qquad \qquad (1-2\beta)h^{2\beta} \geq C^2(2\beta-1) h' $$
hence
\begin{equation}
\begin{aligned}
(h^{1-2\beta})' 
&= (1-2\beta)h^{-2\beta}h' \\
&= (2\beta-1)h^{-2\beta}(-h') \\
&\geq \frac{2\beta-1}{C^2} h^{-2\beta}h^{2\beta} = \frac{2\beta-1}{C^2}
\end{aligned}
\end{equation}
whenever $t\geq t_0$. It follows that there are constants $\alpha \in (0,+\infty)$, $a\in\R$ such that $h^{1-2\beta} \geq \alpha t + a$ holds for all $t\geq t_0$. The map $g:(0,+\infty) \to (0,+\infty)$, $g(y) = y^{\frac{1}{1-2\beta}}$ is strictly decreasing.
Hence $$ \forall t\geq t_0: \qquad \qquad h(t) = g(h(t)^{1-2\beta}) \leq g(\alpha t+a) = (\alpha t+a)^{\frac{1}{1-2\beta}} $$
and finally we arrive at $h(t)t^{1+\epsilon} \leq (\alpha t+a)^{\frac{1}{1-2\beta}}t^{1+\epsilon} \to 0$ as $t\to+\infty$ since $\epsilon < \epsilon_*$ implies $\frac{1}{1-2\beta} + 1 + \epsilon < 0$. This concludes the proof of~\eqref{decay_h}.

Choose $\epsilon<\epsilon'<\epsilon_*$. 
Note that $\|\dot w(t)\| = \|\nabla f(w(t))\| = \sqrt{-h'(t)}$.
Hence, the square of the length of the curve $w(t)$ is estimated with the help of Cauchy-Schwartz inequality as follows:
\begin{equation*}
\begin{aligned}
\left( \int_0^{+\infty} \|\dot w(t)\| \ dt  \right)^2
&= \left( \int_0^{+\infty} \sqrt{-h'(t)} \ dt \right)^2 \\
&= \left( \int_0^{+\infty} \sqrt{-h'(t)} \ t^{\frac{1+\epsilon}{2}} \ t^{-\frac{1+\epsilon}{2}}\ dt \right)^2 \\
&\leq \left( \int_0^{+\infty} - h'(t) \ t^{1+\epsilon} \ dt \right) \quad \left( \int_0^{+\infty} t^{-(1+\epsilon)} \ dt \right)
\end{aligned}
\end{equation*}
Hence, to prove that $w(t)$ has finite length we integrate by parts and use~\eqref{decay_h}:
\begin{equation*}
\begin{aligned}
\int_0^{+\infty} -h'(t) \ t^{1+\epsilon} \ dt
&= - \lim_{t\to+\infty} h(t) \ t^{1+\epsilon} \ + \ \int_0^{+\infty} h(t) \ (1+\epsilon) t^{\epsilon} \ dt \\
&= (1+\epsilon) \int_0^{+\infty} h(t) \ t^{1+\epsilon'} \ t^{-1-\epsilon'+\epsilon} \ dt \\
&\leq (1+\epsilon) \ \|h(t) \ t^{1+\epsilon'}\|_{L^\infty} \int_0^{+\infty} t^{-1-\epsilon'+\epsilon} \ dt < +\infty
\end{aligned}
\end{equation*}
This concludes the proof that $w(t)$ has finite length.
If $\omega(w)$ contains at least two points then $w(t)$ has infinite length.
Hence $\omega(w)$ consists of precisely one point, as we wanted to show.
\end{proof}

% It is remarkable that a version of Lemma~\ref{lemma_omega_limit_set_one_pt} holds for secondary $\omega$-limit sets without the real-analyticity assumption. 
% The reason for this, as the proof will demonstrate, is that after blow-up the vector field is approximately real-analytic.

% \begin{lemma}

% \end{lemma}

\subsection{Completing the proof}

According to the calculations from subsection~\ref{ssec_blow_up}, the vector field $\DIST(\cdot,x_0)^{2-k}\nabla f$ gets pulled back by $\phi$ to a vector field on $(0,\rho) \times \S^{d-1}$ that has a smooth extension $X$ to $(-\rho,\rho) \times \S^{d-1}$. The flow of~$-X$ is denoted here by~$\psi^t$.

The domain of definition of $\Psi = \psi^1$ is an open set $\mathcal{U} \subset (-\rho,\rho)\times \S^{d-1}$. Hence $\Psi : \mathcal{U}\to (-\rho,\rho)\times \S^{d-1}$ is an embedding onto the open set $\Psi(\mathcal{U})$. The $n$-fold iteration $\Psi^n$ of $\Psi$ is defined on an open set $\mathcal{U}_n$ defined inductively as $\mathcal{U}_1=\mathcal{U}$, $\mathcal{U}_n = \{ q \in \mathcal{U} : \Psi(q) \in \mathcal{U}_{n-1} \} = \Psi^{-1}(\mathcal{U}_{n-1})$. Then $\Psi^n(\mathcal{U}_n)\subset(-\rho,\rho)\times \S^{d-1}$ is open and $\Psi^n$ defines a diffeomorphism of $\mathcal{U}_n$ onto $\Psi^n(\mathcal{U}_n)$. If $(r_0,u_0)\in(-\rho,\rho)\times \S^{d-1}$ is such that $\psi^t(r_0,u_0)$ is well-defined for all $t\geq0$, then $(r_0,u_0)\in \cap_{n\geq1}\mathcal{U}_n$. It follows that the latter set is non-empty since $X$ has zeros, for instance it vanishes on $\{0\}\times\CRIT(p) \neq\emptyset$.

By Lemma~\ref{lemma_C1_extension} the set of zeros of $X$ on $\{0\}\times \S^{d-1}$ is precisely $\{0\}\times \CRIT(p)$. Moreover, by Definition~\ref{defn_weak_saddle} and Lemma~\ref{lemma_C1_extension}, for every $u_* \in \CRIT(p)$ the linearization $-DX(0,u_*)$ of $-X$ at $(0,u_*)$ has an eigenvalue with strictly positive real part.
By Lemma~\ref{lemma_measure_zero_vectorfields_manifolds} for each $u_* \in \CRIT(p)$ the point $(0,u_*)$ has an open neighborhood $B(u_*) \subset (-\rho,\rho)\times \S^{d-1}$ with the following property: If $E(u_*)$ is the set of points $(r_0,u_0)$ such that $\psi^t(r_0,u_0)$ is defined and satisfies $\psi^t(r_0,u_0) \in B(u_*)$ for all $t \geq 0$ then $E(u_*)$ has measure zero. Hence, since $\#\CRIT(p)<\infty$ by assumption, the set
\[
E_* = \bigcup_{u_* \in \CRIT(p)} E(u_*)
\]
has measure zero. Again by the fact that $\#\CRIT(p)<\infty$, it can be assumed that the $B(u_*)$, $u_*\in\CRIT(p)$, are mutually disjoint. Consider 
\[
Z = \bigcup_{n\geq 0} \Psi^{-n}(E_*) = \bigcup_{n\geq 0} \{ w \in \mathcal{U}_n : \Psi^n(w) \in E_* \} \, .
\]
Then $Z$ has measure zero since it is a countable union of sets of measure zero.

We can choose $R_0\in(0,\rho)$, $\delta>0$ such that 
\begin{equation*}
(r,u) \in [0,R_0]\times \S^{d-1} \setminus \bigcup_{u_* \in \CRIT(p)} B(u_*) \ \Rightarrow \ \DIST(u,\CRIT(p)) \geq 2\delta.
\end{equation*} 
Lemma~\ref{lemma_simple_analysis} provides $R\in(0,R_0)$ such that the following holds: If the trajectory $\psi^t(r_0,u_0)=(r(t),u(t))$ is well-defined and belongs to $(0,R]\times \S^{d-1}$ for all $t\geq 0$, then $\omega(u(t))$ is contained in the $\delta$-neighborhood of $\CRIT(p)$. Thus, for such $(r_0,u_0)$ one finds $T>0$ such that 
\begin{equation*}
t\geq T \ \Rightarrow \ \DIST(u(t),\CRIT(p)) \leq \frac{3\delta}{2} \ \Rightarrow \ \psi^t(r_0,u_0) \in \bigcup_{u_* \in \CRIT(p)} B(u_*) \, .
\end{equation*}
In particular, there is a unique $u_*$ such that $\psi^t(r_0,u_0)\in B(u_*)$ for all $t\geq T$. As a consequence, if $(r_0,u_0) \in (0,R]\times \S^{d-1}$ is such that $\psi^t(r_0,u_0)$ is well-defined and belongs to $(0,R]\times \S^{d-1}$ for all $t\geq 0$ then for some $n \in \N$ we have $\Psi^n(r_0,u_0) \in E_*$, in particular $(r_0,u_0)\in Z$. 

The set $W(x_0) = \exp\big(\phi(Z\cap [0,R)\times \S^{d-1})\big)$ is contained in the $R$-ball~$B$ centered at $x_0$ with respect to the metric $g$, and has zero measure. Let $w_0\in M$ be such that $\varphi^t(w_0)$ is well-defined and belongs to $B$ for all $t\geq0$. Then, if $(r_0,u_0)$ satisfies $\phi(r_0,u_0)=w_0$, $\psi^t(r_0,u_0)$ is well-defined and belongs to $(0,R)\times\S^{d-1}$. Above it was proved that $(r_0,u_0)\in Z$. Hence, $w_0 \in W(x_0)$.

The proof of Theorem~\ref{thm_weak_center_stable} is complete.

\appendix

\section{Center-stable manifold theorem and its consequences}

In this appendix we prove the version of the center-stable manifold theorem we need.
Consider a linear isomorphism $T:\R^m \to \R^m$ and its spectrum~$\sigma(T) \subset \C$.
Choose a number~$\mu>1$ such that 
\begin{equation}
\label{number_mu}
\lambda \in \sigma(T) \Rightarrow |\lambda| \not\in (1,\mu],
\end{equation}
and choose $\theta$ and $\rho$ satisfying $1 < \theta < \rho < \mu$.
By suitably modifying the real Jordan form of $T$, or alternatively by a direct application of~\cite[Chapter~III,~subsection~6.4]{Kato}, we find a decomposition of $\R^m$ into $T$-invariant linear subspaces
\begin{equation}
\label{decomposition_Jordan}
\R^m = E_1 \oplus E_2 \qquad \qquad T(E_j) \subset E_j
\end{equation}
such that $\sigma(T|_{E_1}) = \sigma(T) \cap \D$, $\sigma(T|_{E_2}) = \sigma(T) \setminus \sigma(T|_{E_1})$.
This and~\eqref{number_mu} imply that there exist norms on $E_1$ and $E_2$, both denoted as $\|\cdot\|$ with no fear of ambiguity, such that
\begin{equation}
\label{adapted_norms_subspaces}
\|Tx\| < \theta \|x\| \quad \forall x\in E_1, \qquad \qquad  \|Ty\| > \mu\|y\| \quad \forall y\in E_2.
\end{equation}
From now on we identify $\R^m = E_1 \oplus E_2 \simeq E_1 \times E_2$ and equip it with the so-called adapted norm
\begin{equation}
\label{adapted_norm_total_space}
\|z=(x,y)\| = \max \{\|x\|,\|y\|\}.
\end{equation}

\begin{remark}
Let $(X,d)$, $(X',d')$ be metric spaces. 
The Lipschitz constant $\mathrm{Lip}(h) \in [0,+\infty]$ of a map $h:X \to X'$ is defined by
$$
\mathrm{Lip}(h) = \sup \left\{ \frac{d'(h(x_1),h(x_2))}{d(x_1,x_2)} : x_1,x_2\in X, \ x_1\neq x_2 \right\} \, .
$$
Note that $h$ is Lipschitz continuous precisely when $\mathrm{Lip}(h)<+\infty$.
\end{remark}

Let $S_1 = \{z=(x,y) \in \R^m : \|x\|\geq\|y\|\}$.
In the lemma below $\R^m$, $E_1$ and~$E_2$ are considered as metric spaces with the distance induced by the norm~\eqref{adapted_norm_total_space}.

\begin{lemma}[Theorem~III.2 in~\cite{Shub}]
\label{lemma_Shub}
There exists $\eta>0$ depending only on~$T$ with the following significance. If $f:\R^m \to \R^m$ satisfies $f(0)=0$ and $\mathrm{Lip} (f-T) < \eta$ then there exists $g:E_1\to E_2$ such that $\mathrm{Lip}(g) \leq 1$, $g(0)=0$ and such that the set $W_1 := \cap_{n\geq 0} f^{-n}(S_1)$ satisfies 
\begin{equation}
W_1 = \mathrm{Graph}(g) = \{z = (x,y) \in \R^m : y=g(x)\}.
\end{equation}
The set $W_1$ is alternatively characterized as
\begin{equation}
W_1 = \left\{ z \in \R^m \ : \ \sup _{n\geq0} \rho^{-n}\|f^n(z)\| < +\infty \right\}.
\end{equation}
\end{lemma}

The proof is postponed to the end of this appendix.

\begin{remark}
It can be shown that if $f$ is $C^r$ then $g$ is $C^r$, see~\cite[Chapter~III]{Shub}.
It can also be shown that the intersection of the set $W_1$ with a small neighborhood of $0$ depends only on the germ of $f$ at $0$.
These facts are not necessary for the purposes of this paper.
\end{remark}

We now prove some useful corollaries of Lemma~\ref{lemma_Shub}.

\begin{lemma}
\label{lemma_center_stable_local}
Consider a $C^1$ map $f:U \to \R^m$ defined on an open set $U \subset \R^m$ such that $0 \in U$, $f(0)=0$, $f(U)$ is open and $f:U \to f(U)$ is a diffeomorphism.
If $Df(0)$ has an eigenvalue of absolute value strictly larger than~$1$, then there exists an open neighborhood $B \subset U$ of $0$ such that set $$ \left\{ z\in U \ : \  f^n(z)\in B \ \forall n\geq 0 \right \} $$
% , \ \lim_{n\to+\infty}f^n(z)=0 \right\} $$ 
has Lebesgue measure zero.
\end{lemma}

\begin{proof}
Denote $T=Df(0)$.
Consider a decomposition $\R^m = E_1 \oplus E_2$ as in~\eqref{decomposition_Jordan} into $T$-invariant subspaces, and an adapted norm~$\|\cdot\|$ as in~\eqref{adapted_norm_total_space} satisfying~\eqref{adapted_norms_subspaces} where $1<\theta<\mu$ and $\mu$ satisfies~\eqref{number_mu}.
Choose $\rho \in (\theta,\mu)$.
Fix a smooth bump function $\phi:\R^m \to \R$ such that $\phi$ takes values on $[0,1]$, $\phi$ is identically~$1$ on $B_1(0)$, and $\supp(\phi)$ is a compact subset of $B_2(0)$.
Here $B_r(0) = \{z\in\R^m : \|z\| < r\}$.
For each $n\geq1$ the $n$-fold composition $f^n$ of $f$ has as domain the open subset $U_n \subset U$ defined inductively as $U_1=U$, $U_n = f^{-1}(U_{n-1})=\{ z\in U : f(z) \in U_{n-1}\}$.
Write $f(z) = Tz + h(z)$ where $h$ satisfies $$ \|h(z)\| = o(\|z\|), \ \|Dh(z)\| = o(1) \ \text{as} \ \|z\| \to 0. $$
For all $s>0$ small enough so that $B_{2s}(0) \subset U$ define $f_s:\R^m \to \R^m$ by
\begin{equation}
\begin{aligned} 
& f_s(z) = Tz + \phi(s^{-1}z)h(z) \qquad \text{if $z\in B_{2s}(0)$,} \\
& f_s(z) = Tz \qquad \qquad \qquad \qquad \ \ \text{if $z \not\in B_{2s}(0)$.}
\end{aligned}
\end{equation}
Then $f_s$ is $C^1$ and a direct computation gives
\begin{equation}
\forall z\in B_{2s}(0): \quad Df_s(z) = T + s^{-1}\left<\nabla\phi(s^{-1}z),\cdot\right> h(z) + \phi(s^{-1}z)Dh(z) \, .
\end{equation}
In~$\R^m \setminus B_{2s}(0)$ the map $f_s$ coincides with $T$.
In~$B_{2s}(0)$ we have $\|z\|\leq 2s$ and $$ \|Df_s(z)-T\| = s^{-1} \ o(s) + o(1) = o(1) \ \text{as $s\to0$}. $$
It follows that for $s$ small enough $f_s$ satisfies the assumptions of Lemma~\ref{lemma_Shub}.
Fix such $s>0$.
Let $W_1$ be the set obtained by applying Lemma~\ref{lemma_Shub} to $f_s$.
Note that the Lebesgue measure of $W_1$ is zero since it is the graph of a Lipschitz function.
Let $z\in U$ satisfy $f^n(z)\in B_s(0)$ for all $n\geq 0$.
Then $f^n(z) = f^n_s(z)$ is bounded.
Lemma~\ref{lemma_Shub} implies that $z\in W_1$.
The desired conclusion follows.
\end{proof}

The version of the above lemma for smooth manifolds is an immediate consequence.
Manifolds are here always assumed Hausdorff and second countable.
Recall that a subset $A$ of a smooth manifold is said to have measure zero if for every chart $(U,\varphi)$ the set $\varphi(A\cap U)$ has Lebesgue measure zero.
It follows that countable unions of sets of measure zero have measure zero, subsets of sets of measure zero have measure zero, and sets of measure zero have dense complement.
To check that $A$ has measure zero it suffices to check that for every $p\in A$ there exists a chart $(U,\varphi)$ such that $p\in U$ and $\varphi(A\cap U)$ has Lebesgue measure zero.
Moreover, if $\phi:U\to V$ is a $C^1$ diffeomorphism between open sets $U,V \subset M$ and $A \subset M$ has measure zero then $f^{-1}(A) = \{p\in U:f(p)\in A\}$ has measure zero.

\begin{lemma}
\label{lemma_center_stable_local_manifolds}
Let $M$ be a smooth manifold, $p_0 \in M$, and $f:U \to M$ be a $C^1$ map defined on an open set $U \subset M$ such that $p_0 \in U$, $f(p_0)=p_0$, $f(U)$ is open and $f:U \to f(U)$ is a diffeomorphism.
If $Df(p_0)$ has an eigenvalue of absolute value strictly larger than~$1$, then there exists an open neighborhood $B$ of $p_0$ such that the set $$ \left\{ p\in U \ : \  f^n(p)\in B \ \forall n\geq 0 \right\} $$
% , \ \lim_{n\to+\infty}f^n(p)=p_0 \right\} $$ 
has measure zero.
\end{lemma}

% \begin{proof}
% Let $m$ be the dimension of $M$ at $p_0$.
% Fix a chart $(V,\varphi)$ such that $p_0 \in V$, $\varphi(V)=\R^m$, $\varphi(p_0)=0$.
% Let $$ W = \{ z\in\R^m : \varphi^{-1}(z) \in f^{-1}(V) \} = \varphi(V\cap f^{-1}(V)) $$
% and consider the map $$ F:W \to \R^m \qquad F = \varphi \circ f \circ \varphi^{-1}. $$
% Then $F$ defines a $C^1$ diffeomorphism of $W$ onto the open set $F(W)$, $F(0)=0$ and $DF(0)$ has an eigenvalue of absolute value larger than $1$.
% By Lemma~\ref{lemma_center_stable_local} we know that the set $$ E_0 = \{ z\in W : F^n(z) \in W \ \forall n\geq0, \ F^n(z) \to 0 \} $$ has Lebesgue measure zero.
% Now let $p\in M$ satisfy $f^n(p) \in U \ \forall n\geq 0$ and $f^n(p) \to p_0$.
% We find $n_0 \in\N$ such that $n\geq n_0 \Rightarrow f^n(p) \in V \cap f^{-1}(V)$.
% Hence $\varphi(f^{n_0}(p)) \in E_0$.
% We proved that $p \in \cup_{n\geq 0} f^{-n}(\varphi^{-1}(E_0))$.
% The conclusion follows since each set $f^{-n}(\varphi^{-1}(E_0))$ has measure zero.
% \end{proof}

For flows we have the following statement.

\begin{lemma}
\label{lemma_measure_zero_vectorfields_manifolds}
Let $X$ be a $C^1$ vector field defined on a smooth manifold $M$. Let $p_0\in M$ satisfy $X(p_0)=0$. If $DX(p_0)$ has an eigenvalue $\lambda$ satisfying $\mathrm{Re} \ \lambda >0$ then there exists an open neighborhood $B \subset M$ of $p_0$ with the following property: the set $E$ of points $p\in M$ such that the solution $\gamma(t)$ of $\dot\gamma=X\circ \gamma$, $\gamma(0)=p$ is defined and satisfies $\gamma(t) \in B$ for all $t\in[0,+\infty)$ has measure zero.
\end{lemma}

\begin{proof}
Standard ODE theory gives an open set $\mathcal{D} \subset \R\times M$ and a $C^1$ map $\varphi : \mathcal{D} \to M$ satisfying the following:
For each $p\in M$ the maximal interval of definition~$J_p$ of the solution $\gamma(t)$ of the initial value problem $\dot\gamma=X\circ \gamma$, $\gamma(0)=p$ satisfies $J_p = \{t\in\R : (t,p) \in \mathcal{D}\}$, and $\gamma(t) = \varphi(t,p)$ holds for all $t\in J_p$.
For each $t\in\R$ denote by $\mathcal{D}_t$ the open set $\{p\in M : (t,p) \in \mathcal{D}\}$, and denote by $\varphi^t:\mathcal{D}_t \to M$ the map $p\mapsto \varphi(t,p)$.
Set $U = \mathcal{D}_1$, $f=\varphi^1:U\to V$.
Let $\lambda$ be an eigenvalue of $DX(p_0)$ such that $\mathrm{Re} \ \lambda>0$.
Then $e^\lambda$ is an eigenvalue of $Df(p_0)$ satisfying $|e^\lambda|>1$.
By Lemma~\ref{lemma_center_stable_local_manifolds} there exists an open neighborhood $B$ of $p_0$ such that the set $E_0$ consisting of points $p\in M$ such that $f^n(p)$ is defined and satisfies $f^n(p) \in B$ for all $n\geq0$ has measure zero.
Use $B$ to define a set $E$ as in the statement.
Then $E \subset E_0$.
The desired conclusion follows.
\end{proof}

We close this appendix with a proof of Lemma~\ref{lemma_Shub} adapted from~\cite{Shub}.

\begin{proof}[Proof of Lemma~\ref{lemma_Shub}]
The proof presented here is based on~\cite[Chapter~III]{Shub}.

Note that $f$ must be Lipschitz since so are $f-T$ and $T$.
We claim that if $2\eta<\|T^{-1}\|^{-1}$ then $f$ is a homeomorphism, $f^{-1}$ is a Lipschitz map and 
\begin{equation}
\label{Lip_f_inverse}
\mathrm{Lip}(f^{-1}) < 2\|T^{-1}\|, \qquad 
\mathrm{Lip}(f^{-1}-T^{-1}) < 2\|T^{-1}\|^2 \ \mathrm{Lip}(f-T).
\end{equation}
The estimate
\begin{equation}
\label{estimate_inverse_f}
\begin{aligned}
\|f(z_2)-f(z_1)\| 
&= \|T(z_2-z_1) + (f-T)(z_2) - (f-T)(z_1)\| \\
&\geq \|T(z_2-z_1)\| - \|(f-T)(z_2) - (f-T)(z_1)\| \\
&\geq (\|T^{-1}\|^{-1}-\eta)\|z_2-z_1\| \\
&> (2\|T^{-1}\|)^{-1} \|z_2-z_1\|
\end{aligned}
\end{equation}
shows that $f$ is injective.
Let $z_0\in \R^m$ and consider $\Phi = id - T^{-1}\circ f + T^{-1}z_0$.
Then $\mathrm{Lip}(\Phi) = \mathrm{Lip}(T^{-1}\circ (T-f)) \leq \|T^{-1}\|\eta<\frac{1}{2}$.
Hence $\Phi$ is a contraction, and its unique fixed point $z$ of $\Phi$ satisfies: $z-T^{-1}f(z)+T^{-1}z_0=z \Rightarrow f(z)=z_0$.
This shows that $f$ is a bijection.
Estimate~\eqref{estimate_inverse_f} gives $\mathrm{Lip}(f^{-1}) < 2\|T^{-1}\|$ which is the first inequality in~\eqref{Lip_f_inverse}.
From $T^{-1}-f^{-1} = T^{-1} \circ (f-T) \circ f^{-1}$ one gets the second inequality in~\eqref{Lip_f_inverse} since Lipschitz constants get multiplied under composition.

Consider the set $S_2 = \{(x,y) \in E_1\times E_2 : \|y\|\geq\|x\|\}$.
From now on denote $\hat\eta = \mathrm{Lip}(f^{-1}-T^{-1})$. 
We write in components
$$ f(x,y) = (f_1(x,y),f_2(x,y)), \qquad f^{-1}(x,y) = (\alpha_1(x,y),\alpha_2(x,y)). $$
according to the identification $\R^m \simeq E_1 \times E_2$.

The set of functions
\begin{equation}
X = \{g:E_1\to E_2:g(0)=0, \ \mathrm{Lip}(g)\leq1\}
\end{equation}
can be equipped with the distance
\begin{equation}
d(g_1,g_2) = \sup_{x\in E_1 \setminus\{0\}} \frac{\|g_2(x)-g_1(x)\|}{\|x\|}.
\end{equation}
It is simple to check that $(X,d)$ is a complete metric space.

The remainder of the proof is split into several steps.

\medskip

\noindent \textbf{Step 0.} For every $g\in X$ the function $\alpha_1 \circ id\times g:E_1 \to E_1$ is a Lipschitz homeomorphism satisfying 
\begin{equation}
\label{Lip_h_inverse}
\mathrm{Lip}((\alpha_1 \circ id\times g)^{-1}) \leq (\theta^{-1}-\hat\eta)^{-1}.
\end{equation}

\medskip

Denote $h=\alpha_1\circ id\times g$.
Then $\mathrm{Lip}(h) \leq \mathrm{Lip}(\alpha_1) \mathrm{Lip}(id\times g) \leq \mathrm{Lip}(f^{-1})<+\infty$.
% $$ \mathrm{Lip}(h) \leq \mathrm{Lip}(\alpha_1) \mathrm{Lip}(id\times g) \leq \mathrm{Lip}(f^{-1})<+\infty. $$
Thus $h$ is Lipschitz.
If $(x,y) \in S_1$ then 
$$ \|T^{-1}y\| < \mu^{-1}\|y\| \leq \mu^{-1}\|x\| < \theta^{-1}\|x\| < \|T^{-1}x\| \, . $$
Hence
$$ (x,y) \in S_1 \Rightarrow \|T^{-1}(x,y)\| = \max \{\|T^{-1}x\|,\|T^{-1}y\|\} = \|T^{-1}x\| > \theta^{-1}\|x\|. $$ 
Moreover, since $\mathrm{Lip}(g)\leq 1$ we get for all $x_1,x_2\in E_1$
$$ (x_2,g(x_2))-(x_1,g(x_1)) \in S_1, \qquad \|(x_2,g(x_2))-(x_1,g(x_1))\| = \|x_2-x_1\| \, . $$
Using these facts we estimate
\begin{equation}
\label{Lip_cte_inverse_h}
    \begin{aligned}
        & \|h(x_2)-h(x_1)\| \\
        &= \|(h(x_2)-h(x_1),g(h(x_2))-g(h(x_1)))\| \\
        &= \|(h(x_2),g(h(x_2)))-(h(x_1),g(h(x_1)))\| \\
        &= \|f^{-1}(x_2,g(x_2)) - f^{-1}(x_1,g(x_1))\| \\
        &\geq \|T^{-1}(x_2,g(x_2))-T^{-1}(x_1,g(x_1))\| \\
        &\qquad - \|(f^{-1}-T^{-1})(x_2,g(x_2)) - (f^{-1}-T^{-1})(x_1,g(x_1))\| \\
        &> (\theta^{-1} - \hat\eta) \|x_2-x_1\|.
    \end{aligned}
\end{equation}
It follows that $h$ is injective. 
To prove surjectivity of $h$ let $x_0 \in E_1$ be fixed arbitrarily, and consider the map $\Phi:\R^m\to\R^m$, $\Phi=I-T\circ f^{-1} + (Tx_0,0)$.
From $I-T\circ f^{-1} = T \circ (T^{-1}-f^{-1})$ we get $\mathrm{Lip}(\Phi) \leq \|T\|\hat\eta$.
Let $\Lambda:E_1\to\R^m$, $\Lambda(x) = \Phi \circ id\times g$, written in components as $\Lambda=(\Lambda_1,\Lambda_2)$.
Then $\mathrm{Lip}(\Lambda_1) \leq \mathrm{Lip}(\Lambda) \leq \mathrm{Lip}(\Phi) \, \mathrm{Lip}(id\times g) \leq \|T\|\hat\eta$.
If $\hat\eta$ is small enough so that $\|T\|\hat\eta<1$ then $\Lambda_1$ is a contraction and has a unique fixed point $x$.
Hence $$ x=\Lambda_1(x) = x - T\alpha_1(x,g(x)) + Tx_0 \Rightarrow h(x)=x_0 $$ as desired.
We are done showing that $h$ is a Lipschitz map and a homeomorphism. 
Estimate~\eqref{Lip_h_inverse} follows from~\eqref{Lip_cte_inverse_h}

\medskip

\noindent \textbf{Step 1.} There exists a unique $g\in X$ such that $$ \forall x\in E_1 : \qquad f^{-1}(x,g(x)) = (h(x),g(h(x))) $$ where $h:E_1 \to E_1$ is the function defined by $h(x) = \alpha_1(x,g(x))$.

\medskip

The proof is based the so-called graph transform: for any $g\in X$ define
\begin{equation}
\label{graph_transform}
\Gamma(g):E_1\to E_2 \qquad \Gamma(g) = (\alpha_2\circ id\times g) \circ (\alpha_1\circ id\times g)^{-1}
\end{equation}
where $id\times g$ denotes the map $x \mapsto (x,g(x))$. 
By Step 0 the map $\alpha_1\circ id\times g$ is a homeomorphism, hence $\Gamma(g)$ is well-defined.
We get from the identity $\alpha_2(x,g(x)) = T^{-1}g(x) + (f^{-1}-T^{-1})_2(x,g(x))$ that 
$$ 
\begin{aligned}
&\|\alpha_2(x_2,g(x_2)) - \alpha_2(x_1,g(x_1))\| \\
&\leq \|T^{-1}(g(x_2)-g(x_1))\| + \hat\eta \|(x_2-x_1,g(x_2)-g(x_1))\| \\
&< \mu^{-1}\|g(x_2)-g(x_1)\| + \hat\eta \max\{\|x_2-x_1\|,\|g(x_2)-g(x_1)\|\} \\
&\leq (\mu^{-1} + \hat\eta)\|x_2-x_1\|.
\end{aligned}
$$
Hence if $\eta$ is small enough then 
$$ \mathrm{Lip}(\Gamma(g)) \leq \mathrm{Lip}(\alpha_2 \circ id\times g) \ \mathrm{Lip}((\alpha_1\circ id\times g)^{-1}) \leq \frac{\mu^{-1}+\hat\eta}{\theta^{-1}-\hat\eta} <1. $$
We have shown that $\Gamma$ defines a map 
$$
\Gamma:X\to X
$$
and now we prove this map is a contraction in $(X,d)$ provided $\eta$ is small enough so that $\mu^{-1}+2\hat\eta < \theta^{-1} - \hat\eta$.
This is a consequence of the fact that
\begin{equation}
\label{technical_estimate_center_stable_1_inverse}
\forall (x,y) \in S_1: \ \frac{\|\alpha_2(x,y)-\Gamma(g)(\alpha_1(x,y))\|}{\|\alpha_1(x,y)\|} \leq \left( \frac{\mu^{-1}+2\hat\eta}{\theta^{-1} - \hat\eta} \right) \frac{\|y-g(x)\|}{\|x\|}
\end{equation} 
holds for all $g\in X$.
Indeed, substitute $y=\hat g(x)$ in~\eqref{technical_estimate_center_stable_1_inverse} to get
$$
\begin{aligned}
d(\Gamma(\hat g),\Gamma(g)) 
&= \sup_{x\in E_1} \frac{\|\Gamma(\hat g)(x)-\Gamma(g)(x)\|}{\|x\|} \\
&= \sup_{x\in E_1} \frac{\|\Gamma(\hat g)(\alpha_1(x,\hat g(x)))-\Gamma(g)(\alpha_1(x,\hat g(x)))\|}{\|\alpha_1(x,\hat g(x))\|} \\ 
&= \sup_{x\in E_1} \frac{\|\alpha_2(x,\hat g(x))-\Gamma(g)(\alpha_1(x,\hat g(x)))\|}{\|\alpha_1(x,\hat g(x))\|} \\
&\leq \left( \frac{\mu^{-1}+2\hat\eta}{\theta^{-1} - \hat\eta} \right) \sup_{x\in E_1} \frac{\|\hat g(x)-g(x)\|}{\|x\|} 
= \left( \frac{\mu^{-1}+2\hat\eta}{\theta^{-1} - \hat\eta} \right) d(\hat g,g).
\end{aligned}
$$

We now prove~\eqref{technical_estimate_center_stable_1_inverse}. 
Fix $(x,y) \in S_1$ arbitrarily.
From
\begin{equation*}
\begin{aligned}
& f^{-1}(x,y) - f^{-1}(x,g(x)) \\
&= (0,T^{-1}y-T^{-1}g(x)) + (f^{-1}-T^{-1})(x,y) - (f^{-1}-T^{-1})(x,g(x))
\end{aligned}
\end{equation*}
we get
\begin{equation}
\label{estimate_alpha_1_S_1}
\|\alpha_1(x,y)-\alpha_1(x,g(x))\| \leq \hat\eta \|y-g(x)\|
\end{equation}
and
\begin{equation}
\label{estimate_alpha_2_S_1}
\begin{aligned}
&\|\alpha_2(x,y)-\alpha_2(x,g(x))\| \\
&\leq \|T^{-1}(y-g(x)) \| + \|(f^{-1}-T^{-1})(x,y) - (f^{-1}-T^{-1})(x,g(x))\| \\
&< (\mu^{-1}+\hat\eta)\|y-g(x)\|.
\end{aligned}
\end{equation}
We also estimate
\begin{equation}
\label{estimate_alpha_1_S_1_second}
\begin{aligned}
\|\alpha_1(x,y)\| &= \|T^{-1}x + \alpha_1(x,y) - T^{-1}x\| \\
&\geq \|T^{-1}x\| - \|\alpha_1(x,y) - T^{-1}x\| \\
&> \theta^{-1}\|x\| - \hat\eta \|(x,y)\| \\
&= (\theta^{-1} - \hat\eta) \|x\|
\end{aligned}
\end{equation}
Putting~\eqref{estimate_alpha_1_S_1} and~\eqref{estimate_alpha_2_S_1} together gives
\begin{equation}
\begin{aligned}
& \|\alpha_2(x,y)-\Gamma(g)(\alpha_1(x,y))\| \\
&= \|\alpha_2(x,y)-\alpha_2(x,g(x)) + \alpha_2(x,g(x)) -\Gamma(g)(\alpha_1(x,y))\| \\
&= \|\alpha_2(x,y)-\alpha_2(x,g(x)) + \Gamma(g)(\alpha_1(x,g(x))) -\Gamma(g)(\alpha_1(x,y))\| \\
&\leq \|\alpha_2(x,y)-\alpha_2(x,g(x))\| + \mathrm{Lip}(\Gamma(g))\|\alpha_1(x,g(x)) - \alpha_1(x,y)\| \\
&< (\mu^{-1}+2\hat\eta)\|y-g(x)\|.
\end{aligned}
\end{equation}
Combining with~\eqref{estimate_alpha_1_S_1_second}
$$ \frac{\|\alpha_2(x,y)-\Gamma(g)(\alpha_1(x,y))\|}{\|\alpha_1(x,y)\|} \leq \left( \frac{\mu^{-1}+2\hat\eta}{\theta^{-1} - \hat\eta} \right) \frac{\|y-g(x)\|}{\|x\|} $$
which is~\eqref{technical_estimate_center_stable_1_inverse}.

Before moving on the next step note that~\eqref{technical_estimate_center_stable_1_inverse} is equivalent to 
\begin{equation}
\label{technical_estimate_center_stable_1}
(x,y) \in f^{-1}(S_1) \Rightarrow \frac{\|y-\Gamma(g)(x)\|}{\|x\|} \leq \left( \frac{\mu^{-1}+2\hat\eta}{\theta^{-1}-\hat\eta} \right)\frac{\|f_2(x,y)-g(f_1(x,y))\|}{\|f_1(x,y)\|}
\end{equation} 
where $g\in X$ is arbitrary.

\medskip

\noindent \textbf{Step 2.} Let $g$ be given by Step 1. 
Then $\mathrm{Graph}(g) = W_1 := \cap_{n\geq0} f^{-n}(S_1)$.

\medskip

From $\mathrm{Lip}(f-T) < \eta$ we have 
\begin{align*}
\|f_1(x,y)\| = \|Tx + f_1(x,y)-Tx\| < \theta\|x\| + \eta\|(x,y)\|, \\
\|f_2(x,y)\| = \|Ty + f_2(x,y)-Ty\| > \mu \|y\| - \eta\|(x,y)\|.
\end{align*}
Hence, if $(x,y)\in S_2$ then
\begin{align}
\label{estimate_f_1_S_2} \|f_1(x,y)\| < (\theta + \eta) \|y\|, \\
\label{estimate_f_2_S_2} \|f_2(x,y)\| > (\mu - \eta) \|y\|.
\end{align}
It follows that if $\eta$ is small enough so that $\theta + \eta < \rho < \mu - \eta$ then 
\begin{equation}
\label{S_2_forward_invariant}
f(S_2) \subset S_2.
\end{equation}
Assume $\eta$ is small enough so that $$ \frac{\mu^{-1}+2\hat\eta}{\theta^{-1}-\hat\eta} < 1. $$
Let $g$ be given by Step 1. 
It satisfies $\Gamma(g)=g$.
With the help~\eqref{technical_estimate_center_stable_1} we will prove
\begin{equation}
\label{point_not_in_graph}
z=(x,y) \in S_1 \setminus \mathrm{Graph}(g) \Rightarrow \exists n_0\geq 0 : f^{n_0}(z) \in S_2.
\end{equation}
Arguing by contradiction, assume that $(x,y) \in S_1 \setminus \mathrm{Graph}(g)$ satisfies $f^n(x,y) = (f^n_1(x,y),f^n_2(x,y)) \in S_1$ for all $n\in\N$.
By~\eqref{technical_estimate_center_stable_1} 
$$ 
\begin{aligned}
\left( \frac{\mu^{-1}+2\hat\eta}{\theta^{-1}-\hat\eta} \right)^{-n} \frac{\|y-g(x)\|}{\|x\|} 
&< \frac{\|f_2^n(x,y)-g(f_1^n(x,y))\|}{\|f_1^n(x,y)\|} \\
& \leq \frac{\|f_2^n(x,y)\|}{\|f_1^n(x,y)\|} + \frac{\|g(f_1^n(x,y))\|}{\|f_1^n(x,y)\|} \leq \frac{\|f_2^n(x,y)\|}{\|f_1^n(x,y)\|} + 1 . 
\end{aligned} $$
Hence $\|f_2^n(x,y)\| > \|f_1^n(x,y)\|$ if $n$ is large enough.
This contradiction proves~\eqref{point_not_in_graph}.
Together~\eqref{S_2_forward_invariant} and~\eqref{point_not_in_graph} imply that $\cap_{n\geq 0} f^{-n}(S_1) \subset \mathrm{Graph}(g)$.
The inclusion $\mathrm{Graph}(g) \subset \cap_{n\geq 0} f^{-n}(S_1)$ follows from the $f$-invariance of $\mathrm{Graph}(g) \subset S_1$, which is a consequence of $$ \forall x\in E_1: \quad f^{-1}(x,g(x)) = (h(x),g(h(x))), \quad f(x,g(x)) = (h^{-1}(x),g(h^{-1}(x))) $$ where $h = \alpha_1 \circ id\times g$.

\medskip

\noindent \textbf{Step 3.} $z \in W_1 \Leftrightarrow \sup_{n\geq0} \rho^{-n}\|f^n(z)\| < +\infty$.
\medskip

By~\eqref{Lip_h_inverse} and $h(0)=0$ 
$$ 
\begin{aligned}
\rho^{-n}\|f^n(x,g(x))\| 
&= \rho^{-n}\|(h^{-n}(x),g(h^{-n}(x)))\| \\
&= \rho^{-n}\|h^{-n}(x)\| \\
&\leq (\rho (\theta^{-1}-\hat\eta))^{-n}\|x\| \to 0
\end{aligned}
$$
for an arbitrary $x\in E_1$, provided $\eta$ is small enough so that $\theta^{-1}-\hat\eta > \rho^{-1}$.

We now show that 
\begin{equation}
\label{not_in_graph}
(x,y) \not\in \mathrm{Graph}(g) \ \Rightarrow \ \rho^{-n}\|f^n(x,y)\| \to +\infty.
\end{equation}
From $f(S_2) \subset S_2$ and~\eqref{estimate_f_2_S_2}
$$
(x,y) \in S_2 \qquad \Rightarrow \qquad \|f^k(x,y)\| \geq (\mu-\eta)^k \|(x,y)\| \ \forall k\geq 0.
$$
By Step 2, if $(x,y) \not\in \mathrm{Graph}(g)$ then there exists $n_0$ such that $f^{n_0}(x,y) \in S_2$. Hence $f^{n_0+k}(x,y) \in S_2$ for all $k\geq 0$, $$ \rho^{-k}\|f^{n_0+k}(x,y)\| \geq (\rho^{-1}(\mu-\eta))^k \ \|f^{n_0}(x,y)\| \to +\infty \ \text{as} \ k\to+\infty $$ and~\eqref{not_in_graph} follows.
\end{proof}

\section*{Acknowledgements}

We gratefully acknowledge funding by the Deutsche Forschungsgemeinschaft (DFG, German Research Foundation) -- Project number 442047500 through the Collaborative Research Center ``Sparsity and Singular Structures'' (CRC 1481). UH acknowledges support of the FAPERJ Grant ``Professor Visitante'' (203.664/2025), and thanks IMPA (Rio de Janeiro) and Henrique Bursztyn for the hospitality. UH thanks Alberto Abbondandolo for helpful technical discussions.

\printbibliography

\end{document}